\documentclass[12pt, reqno]{amsart}
\usepackage{amsmath,amsopn,amssymb,amsthm}
\usepackage[breaklinks=true, colorlinks=true,linkcolor=blue,citecolor=blue,urlcolor=blue]{hyperref}
\usepackage[cp1251]{inputenc}
\usepackage[english]{babel}

\renewenvironment{proof}[1][Proof]{\textbf{#1.} }{\ \rule{0.5em}{0.5em}}

\DeclareMathOperator{\ad}{ad}
\DeclareMathOperator{\Ad}{Ad}

\DeclareMathOperator{\const}{const}
\DeclareMathOperator{\can}{can}
\DeclareMathOperator{\diag}{diag}

\DeclareMathOperator{\Id}{Id}
\DeclareMathOperator{\im}{Im}

\DeclareMathOperator{\kerr}{Ker}
\DeclareMathOperator{\Lin}{Lin}

\DeclareMathOperator{\rk}{rk}

\newtheorem{theorem}{Theorem}
\newtheorem{pred}{Proposition}
\newtheorem{lemma}{Lemma}
\newtheorem{corollary}{Corollary}

\newtheorem{problem}{Problem}

\newtheorem{remark}{Remark}
\newtheorem{example}{Example}

\newtheorem{vopros}{Question}

\begin{document}

\title
[Killing vector fields of constant length on \dots]
{Killing vector fields of constant length on compact homogeneous Riemannian manifolds}
\
\author{Yu.\,G.~Nikonorov}
\address{Yu.\,G. Nikonorov \newline
South Mathematical Institute of  Vladikavkaz Scientific Centre \newline
of the Russian Academy of Sciences, Vladikavkaz, Markus st. 22, \newline
362027, Russia}
\email{nikonorov2006@mail.ru}

\begin{abstract}
In this paper we present some structural results on the Lie algebras of transitive isometry groups of a general
compact homogenous Riemannian manifold with nontrivial Killing vector fields of constant length.

\vspace{2mm}

\vspace{2mm} \noindent Key words and phrases: Clifford-Wolf homogeneous spaces, geodesic orbit spaces, homogeneous spaces, Hermitian symmetric spaces,
homogeneous Riemannian manifolds, Killing vector fields of constant length.
\end{abstract}

\maketitle


\section*{Introduction and the main results}

Recall that a vector field
$X$ on a Riemannian manifold $(M,g)$ is called {\it Killing} if $L_X
g=0$. In this paper, we study {\it Killing vector fields of constant length} on homogeneous Riemannian manifolds.
A detailed study of Killing vector fields of constant length was started in the papers \cite{BerNik2008, BerNik2008n, BerNik2008nn},
although many of the results in this direction have long been known (see a detailed exposition in \cite{BerNik2008}).
It should be noted that there exists a connection between Killing vector fields of constant
length and Clifford-Wolf translations in a Riemannian manifold $(M,g)$.

Recall that a {\it Clifford-Wolf translation} in $(M,g)$ is an
isometry $s$ moving all points in $M$ one and the same distance,
i.~e. $\rho_g\bigl(x,s(x)\bigr) \equiv \const$ for all $x\in M$, where $\rho_g$ means
the inner (length) metric generated by the Riemannian metric
tensor $g$ on $M$. Clifford-Wolf
translations naturally appear in the investigation of homogeneous
Riemannian coverings of homogeneous Riemannian manifolds
\cite{KN,WolfBook}. Clifford-Wolf translations are studied in various papers
(see e.g. \cite{BerNik2008n,F, Wolf61,Wolf62,Ozols1974} for the Riemannian case and \cite{XuDeng2013, XuDeng2014, XuDeng2014n} for
the Finsler case), for a detailed discussion we refer to \cite{BerNik2009} and~\cite{XuWolf2014}.

If a one-parameter isometry group $\gamma(t)$ on $(M,g)$, generated
by a Killing vector field $Z$, consists of Clifford-Wolf
translations, then $Z$ obviously has constant length. This assertion
can be partially inverted: If a Riemannian manifold $(M,g)$ has the injectivity radius,
bounded from below by some positive constant (in particularly, this
condition is satisfied for every compact or homogeneous
manifold), and $Z$ is a Killing vector field of constant
length on $(M,g)$, then the isometries $\gamma(t)$ from the 1-parameter isometry
group, generated by the vector field $Z$, are Clifford-Wolf
translations at least for sufficiently small $|t|$ \cite{BerNik2008n}.

A metric space $(M,\rho)$ is \textit{Clifford-Wolf homogeneous}
if for any points $x, y\in M$ there exists an isometry $f$, Clifford-Wolf translation, of the space $(M,\rho)$ onto itself such that $f(x)=y$.
A connected Riemannian manifold $(M,g)$ is  Clifford-Wolf homogeneous if it
is Clifford-Wolf homogeneous relative to its inner metric $\rho_{g}$.
In addition, it is \textit{$G$-Clifford-Wolf homogeneous} if one can take isometries $f$ from the Lie (sub)group $G$ of isometries of $(M,g)$
in the above definition of Clifford-Wolf homogeneity.
Clifford-Wolf homogeneous simply connected Riemannian manifold are classified in \cite{BerNik2009}:
A~simply connected Riemannian manifold is Clifford-Wolf
homogeneous if and only if it is a direct metric
product of an Euclidean space, odd-dimensional spheres of constant
curvature and simply connected compact simple Lie groups supplied
with bi-invariant Riemannian metrics.
Note that every geodesic $\gamma$ in a Clifford-Wolf
homogeneous Riemannian manifold $(M,g)$ is an
integral curve of a Killing vector field of constant length on~$(M,g)$~\cite{BerNik2009}.

In a recent paper \cite{XuWolf2014}, Ming Xu and Joseph A. Wolf obtained the classification of normal Riemannian homogeneous spaces $G/H$ with nontrivial
Killing vector fields of constant length, where $G$ is compact and simple. Every of these spaces with~$\dim(G)>\dim(H)>0$ is locally symmetric
and its universal Riemannian cover is either an odd-dimensional sphere of constant curvature, or a Riemannian symmetric space $SU(2n)/Sp(n)$.
This result is very important in the context of the study of general Riemannian homogeneous manifolds with nonzero Killing fields of constant length.
In a very recent paper \cite{WolfPodXu2015}, this result was extended to the class of pseudo-Riemannian normal homogeneous spaces.
\smallskip

In this paper we present some structural results on the Lie algebras of transitive isometry groups of a general
compact homogenous Riemannian manifold with
nontrivial Killing vector fields of constant length.
All manifolds supposed connected throughout this paper.
\smallskip

Let us consider any Lie group $G$ acting on the Riemannian manifold $(M,g)$ by isometries.
The action of $a$  on $x\in M$ will be denoted by $a(x)$.
We will identify the Lie algebra $\mathfrak{g}$ of $G$ with the corresponded Lie algebra of Killing vector field on $(M,g)$ as follows.
For any $U\in \mathfrak{g}$ we consider a one-parameter group $\exp(tU)\subset G$ of isometries of $(M,g)$ and define a Killing vector field $\widetilde{U}$
by a usual formula
\begin{equation}\label{kfaction}
\widetilde{U} (x) = \left.\frac{d}{dt}\exp(tU)(x)\right|_{t=0} .
\end{equation}
It is clear that the map $U \rightarrow \widetilde{U}$ is linear and injective, but $[\widetilde{U},\widetilde{V}]= -\widetilde{[U,V]}$.
We will use this identification repeatedly in this paper.

Let $(M,g)$ be a compact connected Riemannian manifold, $G$ is a transitive isometry group of $(M,g)$.
We identify elements of the Lie algebra $\mathfrak{g}$ of $G$ with Killing vector fields on $(M,g)$ as above.
Since $G$ is compact, then we have a decomposition

\begin{equation}\label{str1}
\mathfrak{g}= \mathfrak{c}\oplus \mathfrak{g}_1 \oplus \mathfrak{g}_2 \oplus \cdots \oplus \mathfrak{g}_k,
\end{equation}
where
$\mathfrak{c}$ is the center and $\mathfrak{g}_i$, $i=1,\dots,k$, are simple ideals in $\mathfrak{g}$.

We are going to state the main results of this paper, that are formulated under the above assumptions and notations.

\begin{theorem}\label{theorem1}
Let $Z =Z_0+Z_1+Z_2+\cdots + Z_l\in \mathfrak{g}$ be a Killing vector field of constant length on $(M,g)$, where
$1\leq l \leq k$, $Z_0\in \mathfrak{c}$, $Z_i \in \mathfrak{g}_i$ and $Z_i \neq 0$ for $1\leq i \leq l$.
Then the following statements hold:

{\rm 1)} For every $i\neq j$, $1\leq i,j \leq l$, we have $g(\mathfrak{g}_i, \mathfrak{g}_j)=0$ at every point of $M$.
In particular, $g(Z_i, \mathfrak{g}_j)=0$ and $g(Z_i,Z_j)=0$.

{\rm 2)} Every Killing field of the type $Z_0+Z_i$,  $1\leq i \leq l$, has constant length.

Conversely, if for every $i$,  $1\leq i \leq l$, the Killing field $Z_0+Z_i$ has constant length, and for every
$i\neq j$, $1\leq i,j \leq l$, the equality $g(\mathfrak{g}_i, \mathfrak{g}_j)=0$ holds, then
the Killing field $Z=Z_0+Z_1+Z_2+\cdots + Z_l$ has constant length on $(M,g)$.
\end{theorem}

\begin{corollary}\label{cortheo1}
If {\bf(}under the assumptions of theorem \ref{theorem1}{\bf)} we have $\mathfrak{c}=0$ and $k=l$, then every $\mathfrak{g}_i$, $1\leq i\leq k$,
is a parallel distribution on $(M,g)$. Moreover, if $(M,g)$ is simply connected, then it is a direct metric product of $k$ Riemannian manifolds.
\end{corollary}

Theorem \ref{theorem1} allows to restrict our attention on Killing vector fields of constant length of the following special type:
$Z=Z_0+Z_i$, where $Z_0$ is in the center $\mathfrak{c}$ of $\mathfrak{g}$
and $Z_i$ is in the simple ideal $\mathfrak{g}_i$ in $\mathfrak{g}$. Without loss of generality we will assume that $i=1$.

\begin{theorem}\label{theorem2}
Let $Z =Z_0+Z_1\in \mathfrak{g}$ be a Killing vector fields of constant length on $(M,g)$, where
$Z_0\in \mathfrak{c}$, $Z_1 \in \mathfrak{g}_1$ and $Z_1 \neq 0$, and let $\mathfrak{k}$ be the centralizer of $Z$ {\rm(}and $Z_1${\rm)} in $\mathfrak{g}_1$.
Then either any $X\in \mathfrak{g}_1$ is a Killing field of constant length on $(M,g)$, or
the pair $(\mathfrak{g}_1, \mathfrak{k})$
is one of the following irreducible Hermitian symmetric pair:

{\rm 1)} $(su(p+q), su(p)\oplus su(q)\oplus \mathbb{R})$, $p\geq q \geq 1$;

{\rm 2)} $(so(2n),su(n)\oplus \mathbb{R})$, $n\geq 5$;

{\rm 3)} $(so(p+2), so(p)\oplus \mathbb{R})$, $p \geq 5$;

{\rm 4)} $(sp(n),su(n)\oplus \mathbb{R})$, $n\geq 2$;

\noindent In the latter four cases the center of $\mathfrak{k}$ is a one-dimensional Lie algebra spanned by the vector~$Z_1$.
\end{theorem}

If a Killing field of constant length $Z=Z_0+Z_1$ satisfies one of the cases 1)-4) in theorem \ref{theorem2}, we will say that it has {\it Hermitian type}.
Recall that an element $U\in \mathfrak{g}$ is {\it regular} in $\mathfrak{g}$, if
its centralizer has minimal dimension among all the elements of $\mathfrak{g}$.
For~$Z$ of Hermitian type, $Z_1$ is not
a regular element in $\mathfrak{g}_1$, since $\mathfrak{k}$ is not commutative in the cases 1)-4) of theorem \ref{theorem2}.
Hence, we get

\begin{corollary}\label{cortheo2}
If $Z_1$ is a regular element in the Lie algebra $\mathfrak{g}_1$ {\rm(}under the assumptions of theorem {\ref{theorem2} \rm)},
then every $X\in \mathfrak{g}_1$ is a Killing vector field of constant length on $(M,g)$.
\end{corollary}

Moreover, the following result holds.

\begin{theorem}\label{theorem3}
If $Z=Z_0+Z_1+\cdots +Z_k$ is a regular element of $\mathfrak{g}$ and
has constant length on $(M,g)$, then the following assertions hold:

{\rm 1)} $g(\mathfrak{g}_i,\mathfrak{g}_j)=0$ for every $i\neq j$, $i,j=1,\dots,k$, at every point of $M$;

{\rm 2)} $g(Z_0,\mathfrak{g}_i)=0$ for every $i=1,\dots,k$, at every point of $M$;

{\rm 3)} every  $X\in \mathfrak{g}_{\mathfrak{s}}$,
where $\mathfrak{g}_{\mathfrak{s}}=\mathfrak{g}_1\oplus\cdots\oplus\mathfrak{g}_k$
is a semisimple part of $\mathfrak{g}$, has constant length on $(M,g)$.

Moreover, if $M$ is simply connected, then $Z_0=0$ and $(M,g)$ is a direct metric product of $(G_i,\mu_i)$,
$1\leq i \leq k$, where $G_i$ is a connected and simply connected compact simple Lie group
with the Lie algebra $\mathfrak{g}_i$ and $\mu_i$ is a bi-invariant Riemannian metric on $G_i$.
\end{theorem}
\medskip

The structure of the paper is the following.  We consider various examples of Killing vector fields of constant length  in section 1.
Sections 2, 3, and 4 are devoted to the presentation of auxiliary results on Killing vector fields on Riemannian manifolds,
on some algebraic properties of special homogeneous spaces, and on the root spaces decompositions of simple Lie algebras.
In section 5 we prove theorems 1 and 3, corollary 1 and start to prove theorem 2.
We finish the proof of theorem 2 in section~6. Section 7 devoted to a more detailed study of cases 1)-4) in theorem \ref{theorem2}.
Finally, we discuss some unsolved problems and questions in section 8.
\smallskip

{\bf Acknowlegment.}
The author extends gratitude to Professor V.N. Berestovskii for interesting and helpful discussions concerning this project and
to the anonymous referee
for helpful comments and suggestions that improved the presentation of this paper.

\section{Examples of Killing vector fields of constant length}\label{examsec}

In this section we discuss some examples of Killing vector fields of constant length on Riemannian manifolds.
At first we consider examples for theorem \ref{theorem2} with simple~$\mathfrak{g}_1$ consisting of Killing fields of constant length.

It is well known that the group of left translations, as well as the group of right translation, of a compact Lie group $G$,
supplied with a bi-invariant Riemannian metric~$\mu$, consists of Clifford-Wolf translations. Therefore, Killing fields that generate
these groups have constant length on $(G,\mu)$. Of course, every direct metric product of $(G,\mu)$ with any Riemannian manifolds
also has Killing fields of constant length. Now, let us consider more interesting examples.
\smallskip

Let $F$ be a connected compact simple Lie group, $k \in \mathbb{N}$.  Consider a so-called Ledger~--~Obata space $F^k/\diag(F)$, see
\cite[section 4]{LedgerObata} or \cite{Nikonorov3}. We supply it with any invariant Riemannian metric $g$.
The structure of invariant Riemannian metrics on~$F^k/\diag(F)$ is discussed in details in \cite{Nikonorov3}.
It should be noted that for any compact Lie group $F$, a  Ledger~--~Obata space
$F^k/\diag(F)$ is diffeomorphic to the Lie group $F^{k-1}$ \cite[P.~453]{LedgerObata}. It is easy to see that
every copy $F$ in $F^k$ consists of Clifford-Wolf translations on~$(F^k/\diag(F),g)$ as well as
every copy of the Lie algebra $\mathfrak{f}$ in $\mathfrak{f}\oplus \mathfrak{f} \oplus \cdots \oplus \mathfrak{f}=k\cdot \mathfrak{f}$
consists of Killing fields of constant length for any invariant Riemannian metric $g$.
For example, we may consider $g$ induced with the Killing form of $k\cdot \mathfrak{f}$. For such a choice
$(F^k/\diag(F),g)$ is (locally) indecomposable.
\smallskip

The above example could be generalized as follows. Consider a compact simple Lie group $F$ and compact Lie groups $G_1,G_2,\dots,G_l$ such that
every $G_i$ has a subgroup isomorphic to $F\times K_i$.
Then the group $F\times G_1\times \cdots \times G_l$ acts naturally on $M=G_1/K_1\times \cdots \times G_l/K_l$:
$$
(a, b_1,b_2,\dots,b_l)\cdot (c_1K_1,c_2K_2,\dots, c_lK_l)\rightarrow (b_1c_1K_1a^{-1},b_2c_2K_2a^{-1},\dots, b_lc_lK_la^{-1}).
$$
Hence, $M=(F\times G_1\times \cdots \times G_l)/(K_1\times \cdots \times K_l\times\diag(F))$. It is clear that for any
$(F\times G_1\times \cdots \times G_l)$-invariant Riemannian metric $g$, $F$ consists of Clifford-Wolf translation on $(M,g)$ and every
$Z\in \mathfrak{f}$ has constant length on $(M,g)$, because the group $G_1\times \cdots \times G_l$ is transitive on $(M,g)$.
\medskip

In what follows, we will consider examples for the cases 1)-4) in theorem \ref{theorem2}.
\smallskip

We say that a Lie algebra $\mathfrak{g}$ of Killing vector fields is {\it transitive} on a Riemannian manifold $(M,g)$
if $\mathfrak{g}$ generates the tangent space to $M$ at every point $x\in M$, or, equivalently,
the connected isometry group $G$ with the Lie algebra $\mathfrak{g}$ acts transitively on $(M,g)$.

The following simple observation gives many examples of Killing vector fields of constant length on homogeneous Riemannian manifolds.
Suppose that a Lie algebra $\mathfrak{g}$ of Killing vector fields is transitive on a Riemannian manifold $(M,g)$ and
the Killing field $Z$ on $(M,g)$ commutes with $\mathfrak{g}$ (in particular, $Z$ is in the center of $\mathfrak{g}$),
then $Z$ has a constant length on $(M,g)$. Indeed,
for any $X\in \mathfrak{g}$ we have $X\cdot g(Z,Z)=2g([X,Z],Z)=0$. Since $\mathfrak{g}$ is transitive on $M$, we get
$g(Z,Z)=\const$.

\begin{example}\label{exam1}
Consider the irreducible symmetric space $M=SU(2n)/Sp(n)$, $n \geq 2$. It is known
that the subgroup $SU(2n-1)\cdot S^1 \subset SU(2n)$ acts transitively on $M$, see e.g. \cite{On1994} or \cite{WolfBook}. Therefore,
the Killing vector
$Z$ generated by $S^1$, is a Killing field of constant length on $M=SU(2n)/Sp(n)$.
The centralizer $\mathfrak{k}$ of $Z$ in $\mathfrak{g}_1$ is obviously $su(2n-1)\oplus \mathbb{R}$, i.e.
$(\mathfrak{g}_1,\mathfrak{k})=(su(2n), su(2n-1)\oplus \mathbb{R})$,  see case {\rm 1)} in theorem \ref{theorem2}.
\end{example}

\begin{example}\label{exam2}
Consider the sphere $S^{2n-1}$, $n\geq 2$, as the symmetric space $S^{2n-1}=SO(2n)/SO(2n-1)$. It is known
that the subgroup $U(n)=SU(n)\cdot S^1 \subset SO(2n)$ acts transitively on $S^{2n-1}$. Therefore, the Killing vector
$Z$ generating $S^1$, is a Killing field of constant length on $S^{2n-1}$.
The centralizer $\mathfrak{k}$ of $Z$ in $\mathfrak{g}_1$ is $u(n)=su(n)\oplus \mathbb{R}$, i.e.
$(\mathfrak{g}_1,\mathfrak{k})=(so(2n), su(n)\oplus \mathbb{R})$, see case {\rm 2)} in theorem \ref{theorem2}.
\end{example}

\begin{example}\label{exam3}
Consider the sphere $S^{4n-1}$, $n\geq 2$, as the symmetric space $S^{4n-1}=SO(4n)/SO(4n-1)$. It is known
that the subgroup $Sp(n)\cdot S^1 \subset SO(4n)$ acts transitively on $S^{4n-1}$. Therefore, the Killing vector
$Z$ generated by $S^1$, is a Killing field of constant length on $S^{4n-1}$.
On the other hand, the centralizer $\mathfrak{k}$ of $Z$ in $\mathfrak{g}_1$ is
$su(2n)\oplus \mathbb{R}$, since $sp(n)\oplus \mathbb{R} \subset su(2n)\oplus \mathbb{R}$ {\rm(}see details e.g. in \cite{BerNik2014} {\rm)}.
Therefore, $(\mathfrak{g}_1,\mathfrak{k})=(so(4n), su(2n)\oplus \mathbb{R})$, see case {\rm 2)} in theorem \ref{theorem2}.
\end{example}

Note that the full connected isometry group of the sphere $S^{n-1}$ with the canonical Riemannian metric $g_{\can}$ of constant curvature $1$,
is $SO(n)$, but there are some subgroups $G$ of $SO(n)$ with transitive action on $S^{n-1}$.
It is interesting that $S^{n-1}$ is $G$-Clifford-Wolf homogeneous for some of them:
$S^{2n-1}=SO(2n)/SO(2n-1)=U(n)/U(n-1)$ is $SO(2n)$-Clifford-Wolf homogeneous and $U(n)$-Clifford-Wolf homogeneous;
$S^{4n-1}=Sp(n)/Sp(n-1)=Sp(n)\cdot S^1/Sp(n-1)\cdot S^1=SU(2n)/SU(2n-1)$ is $Sp(n)$-Clifford-Wolf homogeneous, $Sp(n)\cdot S^1$-Clifford-Wolf
homogeneous, and $SU(2n)$-Clifford-Wolf homogeneous;
$S^7=Spin(7)/G_2$ is $Spin(7)$-Clifford-Wolf homogeneous; $S^{15}=Spin(9)/Spin(7)$ is $Spin(9)$-Clifford-Wolf homogeneous
(see details in \cite{BerNik2014}).
Every of these results gives an example of Killing vector field of constant length in the Lie algebra $\mathfrak{g}$ corresponded to the group $G$.

\begin{example}\label{exam4}
There exists a Killing vector field of constant length $Z$ on $(S^{4p-1},g_{\can})$ such that
$Z\in su(2p)$ and
the centralizer $\mathfrak{k}$ of $Z$ in $su(2p)$ is $su(p)\oplus su(p) \oplus \mathbb{R}$, i.e.
$(\mathfrak{g}_1,\mathfrak{k})=(su(2p), su(p)\oplus su(p) \oplus \mathbb{R})$, see case {\rm 1)} in theorem \ref{theorem2} and proposition 12 in~\cite{BerNik2014}.
Note, that we may take $Z=\diag(\underbrace{{\bf i},  \dots, {\bf i}}_{p}\,, \underbrace{{\bf -i},  \dots, {\bf -i}}_{p}\,)\in su(2p)$.
\end{example}

\begin{example}\label{exam5}
There exists a Killing vector field of constant length $Z$ on $(S^{2(p+q)-1},g_{\can})$ such that
$Z\in u(p+q)$ and
the centralizer $\mathfrak{k}$ of $Z$ in $u(p+q)$ is $u(p)\oplus u(q)$, i.e.
$(\mathfrak{g}_1,\mathfrak{k})=(su(p+q), su(p)\oplus su(q) \oplus \mathbb{R})$, see case {\rm 1)} in theorem \ref{theorem2} and proposition 12 in~\cite{BerNik2014}.
Note, that we may take $Z=\diag(\underbrace{{\bf i},  \dots, {\bf i}}_{p}\,, \underbrace{{\bf -i},  \dots, {\bf -i}}_{q}\,)\in u(p+q)$.
Note also that the Lie algebra $u(p+q)$ has a 1-dimensional center, and $Z_0=0$ if and only if $p=q$, when we get the previous example.
If $p\neq q$, then the Killing field $Z_1$ does not have a constant length on $(S^{2(p+q)-1},g_{\can})$ by proposition \ref{newpr7},
see also proposition \ref{supred}.
\end{example}

\begin{example}\label{exam6}
There exists a Killing vector field of constant length $Z$ on $(S^7,g_{\can})$ such that
$Z\in spin(7)\simeq so(7)$ and
the centralizer $\mathfrak{k}$ of $Z$ in $spin(7)\simeq so(7)$ is $so(5)\oplus \mathbb{R}$, i.e.
$(\mathfrak{g}_1,\mathfrak{k})=(so(7), so(5)\oplus \mathbb{R})$, see case {\rm 3)} in theorem \ref{theorem2} and remark 6 in \cite{BerNik2014}.
\end{example}

\begin{example}\label{exam7}
There exists a Killing vector field of constant length $Z$ on $(S^{15},g_{\can})$ such that
$Z\in spin(9)$ and
the centralizer $\mathfrak{k}$ of $Z$ in $spin(9)\simeq so(9)$ is $so(7)\oplus \mathbb{R}$, i.e.
$(\mathfrak{g}_1,\mathfrak{k})=(so(9), so(7)\oplus \mathbb{R})$, see case {\rm 3)} in theorem \ref{theorem2} and proposition 20 in \cite{BerNik2014}.
\end{example}

\begin{example}\label{exam8}
There exists a Killing vector field of constant length $Z$ on $(S^{4n-1},g_{\can})$ such that
$Z\in sp(n)$ and
the centralizer $\mathfrak{k}$ of $Z$ in $sp(n)$ is $u(n)$, i.e.
$(\mathfrak{g}_1,\mathfrak{k})=(sp(n), su(n)\oplus \mathbb{R})$, see case {\rm 4)} in theorem \ref{theorem2} and proposition 12 in \cite{BerNik2014}.
Note, that we may take $Z=\diag({\bf i}, {\bf i},\dots,{\bf i})\in sp(n)$.
\end{example}

\section{Preliminaries on Killing vector fields of constant length }\label{gen0}

Let $(M,g)$ be a connected Riemannian manifolds, $G$ be a connected Lie group acting isometrically on $(M,g)$ ($x\mapsto a(x)$, $x\in M$, $a \in G$).
Let $\mathfrak{g}$ be the Lie algebra of the group $G$, all elements of $\mathfrak{g}$ we identify with Killing fields on $(M,g)$ via (\ref{kfaction}).
We recall some important properties of Killing vector fields, in particular, Killing fields of constant length on $(M,g)$.
We also prove some useful results on Killing vector fields of constant length.

\begin{lemma}[Lemma 3 in \cite{BerNik2009}]\label{lemma1}
Let $X$ be a Killing vector field on a Riemannian manifold
$(M,g)$. Then the following conditions are equivalent:

{\rm 1)} $X$ has constant length on $M$;

{\rm 2)} $\nabla_XX=0$ on $M$;

{\rm 3)} every integral curve of the field $X$ is a geodesic in $(M,g)$.
\end{lemma}

\begin{lemma}\label{le1}
If a Killing vector field $X\in \mathfrak{g}$ has constant length on $(M,g)$, then for any
$Y,Z\in \mathfrak{g}$ the  equalities
\begin{eqnarray}
g([Y,X],X)=0 \,,\label{e0}\\
g([Z,[Y,X]],X)+g([Y,X],[Z,X])=0\,\,\, \label{e00}
\end{eqnarray}
hold at every point of $M$.
If $G$ acts on $(M,g)$ transitively, then condition (\ref{e0}) implies that $X$ has constant length.
Moreover, the condition (\ref{e00}) also implies that $X$ has constant length for compact $M$ and transitive $G$.
\end{lemma}

\begin{proof}
If $g(X,X)=\const$, then $2g([Y,X],X)= Y \cdot g(X,X)=0$ at every point of $M$ for every $Y\in \mathfrak{g}$, that proves (\ref{e0}).
From this we have
$0=Z\cdot g([Y,X],X)= g([Z,[Y,X]],X)+g([Y,X],[Z,X])$ for any $Z\in \mathfrak{g}$, that proves (\ref{e00}).

In the case of transitive action, we obviously get $g(X,X)=\const$ from the equality $Y \cdot g(X,X)=2g([X,Y],X)=0$, $Y\in \mathfrak{g}$.

Note also that in the case of transitive action the condition (\ref{e00}) means that for any $Y\in \mathfrak{g}$ we have
$g([Y,X],X)=C=\const$ on $M$. If $M$ is compact then there is $x\in M$ where $g(X,X)$ has its maximal value. Obviously, that
$g_x([Y,X],X)=0$. Therefore, $C=0$ and $X$ has constant length by the previous assertion.
\end{proof}

\begin{lemma}[see e.g. lemma 7.27 in \cite{Bes}]\label{lemma3}
For Killing vector fields $X,Y,Z$ on a Riemannian manifold $(M,g)$, the equality
$$
2g(\nabla_X Y, Z)= g([X,Y],Z)+g([X,Z],Y)+g(X,[Y,Z])=g([X,Y],Z)-Z\cdot g(X,Y)
$$
holds. In particular, $\nabla_XY=\frac{1}{2}[X,Y]$, if  $g(X,Y)=\const$ and $(M,g)$ is homogeneous.
\end{lemma}

\begin{lemma}\label{lemma5}
Let $X$ be a Killing field and let $Y$ be a Killing field of constant length on a
Riemannian manifold $(M,g)$. Then the formula $R(X,Y)Y=-\nabla_Y\nabla_Y X$ holds on~$M$  {\rm(}$R(X,Z):=\nabla_X \nabla_Z-\nabla_Z \nabla_X-\nabla_{[X,Z]}${\rm)}.
\end{lemma}

\begin{proof}
All integral curve of the field $Y$ are geodesics by lemma \ref{lemma1}.
On the other hand, the restriction of the Killing field $X$ on any geodesic is an Jacobi field (see e.g. proposition 1.3 of chapter VIII in \cite{KN}).
Hence, $\nabla_Y\nabla_Y X+R(X,Y)Y=0$ on $M$ .
\end{proof}
\medskip

We will use the identification (\ref{kfaction})
of elements of Lie algebras $\mathfrak{g}$ of $G$ with corresponding Killing vector fields on $(M,g)$ in the proof of the following

\begin{lemma}\label{s1}
Let $X,Y \in \mathfrak{g}$ be such that $g(X,Y)=C=\const$ on $M$. Then for every inner automorphism $A$ of $\mathfrak{g}$ we
get $g(A(X),A(Y))=C$ at every point of $M$. In particular, if $X\in \mathfrak{g}$ has constant length on $(M,g)$, then
$A(X)$ has the same property.
\end{lemma}

\begin{proof}
If $L_a:M \rightarrow M$ is the action of $a\in G$ on $M$,
then for any $U\in \mathfrak{g}$ we have
\begin{eqnarray*}
dL_a(\widetilde{U} (x))=\left.\frac{d}{dt}\bigl(a\exp(tU)\bigr)(x)\right|_{t=0} =\left.\frac{d}{dt} \bigl(a\exp(tU)a^{-1}\bigr)(a(x))\right|_{t=0}=\\
\left.\frac{d}{dt}\Bigl(\exp\bigl(\Ad(a)(U)t+o(t)\bigr)\Bigr)(a(x))\right|_{t=0} =\widetilde{\Ad(a)(U)}(a(x)).
\end{eqnarray*}
Since $L_a$ is an isometry of $(M,g)$, then
$$
g_x(\widetilde{X},\widetilde{Y})=g_{L_a(x)}(dL_a(\widetilde{X}), dL_a(\widetilde{Y}))=g_{a(x)}(\widetilde{\Ad(a)(X)},\widetilde{\Ad(a)(Y)})
$$
at every point $x\in M$. Recall that the inner automorphism group of $\mathfrak{g}$ coincides with the adjoint group of $G$, therefore we get the lemma.
\end{proof}

\section{Some algebraic lemmas}\label{salsec}

A Lie algebra $\mathfrak{g}$ is called {\it compact}, if it is a Lie algebra of some compact Lie group.
Any such Lie algebra admits an $\ad(\mathfrak{g})$-invariant inner product.
The following lemma is known in the literature, but we include its proof for completeness.

\begin{lemma}\label{vslem1}
Suppose that $\mathfrak{h}$ is a subalgebra of a compact Lie algebra $\mathfrak{g}$, and
$\mathfrak{p}$ is an $\langle \cdot ,\cdot \rangle$-orthogonal complement to
$\mathfrak{h}$ in $\mathfrak{g}$, where
$\langle \cdot, \cdot \rangle$ is an $\ad(\mathfrak{g})$-invariant
inner product on $\mathfrak{g}$.
Then the sets $\mathfrak{p}+ [\mathfrak{p},\mathfrak{p}]$ and
$\mathfrak{u}:=\{Z\in \mathfrak{h}\, |\, [Z,\mathfrak{p}]=0\}$ are ideals in $\mathfrak{g}$.
\end{lemma}

\begin{proof}
Since $[\mathfrak{u},\mathfrak{h}] \subset \mathfrak{h}$ and $\mathfrak{p}$
is $\ad(\mathfrak{h})$-invariant, we get
$$
[\mathfrak{p},[\mathfrak{u},\mathfrak{h}]] \subset [[\mathfrak{p},\mathfrak{u}],\mathfrak{h}]]+
[\mathfrak{u},[\mathfrak{p},\mathfrak{h}]]=0,
$$
hence $[\mathfrak{u},\mathfrak{h}] \subset \mathfrak{u}$. Since $[\mathfrak{u},\mathfrak{p}]=0$, we get that
$\mathfrak{u}$ is an ideals in $\mathfrak{g}$. It is easy to check that
$\mathfrak{p}+ [\mathfrak{p},\mathfrak{p}]$
is the complementary ideal to $\mathfrak{u}$ in $\mathfrak{g}$.
Indeed, for $Z\in \mathfrak{h}$ the condition $\langle Z, [\mathfrak{p},\mathfrak{p}]\rangle =0$ is equivalent
to every of the conditions $\langle [Z,\mathfrak{p}],\mathfrak{p}\rangle =0$ and $[Z,\mathfrak{p}]=0$.
\end{proof}

We will need also the following generalization of the previous lemma (compare e.g. with the proof of Theorem 2.1 in \cite{WZ3}).

\begin{lemma}\label{vslem2} Let $\mathfrak{g}$ be a compact Lie algebra,
$\langle \cdot, \cdot \rangle$ be an $\ad(\mathfrak{g})$-invariant
inner product on $\mathfrak{g}$, $\mathfrak{h}$ be a subalgebra of $\mathfrak{g}$.
Suppose that the $\langle \cdot ,\cdot \rangle$-orthogonal
complement $\mathfrak{p}$ to $\mathfrak{h}$ in $\mathfrak{g}$ is of the type
$$
\mathfrak{p}=\mathfrak{p}_1\oplus \mathfrak{p}_2 \oplus \cdots \oplus \mathfrak{p}_k,
$$
where every $\mathfrak{p}_i$ is $\ad(\mathfrak{h})$-invariant and $\ad(\mathfrak{h})$-irreducible, $\langle \mathfrak{p}_i,\mathfrak{p}_j \rangle=0$ and
$[\mathfrak{p}_i,\mathfrak{p}_j]=0$ for $i\neq j$. For $1\leq i \leq k$ put
$\mathfrak{h}_i:=\bigl(\mathfrak{p}_i+[\mathfrak{p}_i,\mathfrak{p}_i]\bigr) \cap \mathfrak{h}$
and $\mathfrak{g}_i:=\mathfrak{h}_i \oplus \mathfrak{p}_i$. Then the following assertions hold:

{\rm 1)} $\mathfrak{g}_i=\mathfrak{p}_i+[\mathfrak{p}_i,\mathfrak{p}_i]$ for every $1 \leq i \leq k$;

{\rm 2)} every $\mathfrak{g}_i$ is an ideal in $\mathfrak{g}$;

{\rm 3)} $\langle \mathfrak{g}_i,\mathfrak{g}_j \rangle =0$ for $i\neq j$;

{\rm 4)} every pair $(\mathfrak{g}_i, \mathfrak{h}_i)$ is effective and isotropy irreducible;

{\rm 5)} $\mathfrak{u}$, the $\langle \cdot,\cdot \rangle$-orthogonal complement to
$\bigoplus_i \mathfrak{g}_i$ in $\mathfrak{g}$,
is an ideal in $\mathfrak{g}$;

{\rm 6)} $\mathfrak{g}=\mathfrak{u}\oplus \mathfrak{g}_1\oplus \mathfrak{g}_2 \oplus \cdots \oplus \mathfrak{g}_k$,
$\mathfrak{h}=\mathfrak{u}\oplus \mathfrak{h}_1\oplus \mathfrak{h}_2 \oplus \cdots \oplus \mathfrak{h}_k$.
\end{lemma}

\begin{proof}
At first, we prove that $\mathfrak{g}_i=\mathfrak{p}_i+[\mathfrak{p}_i,\mathfrak{p}_i]$ for every $1\leq i \leq k$.
Consider any $j \neq i$, $1\leq j \leq k$.
For every $X,Z \in \mathfrak{p}_i$ and $Y \in \mathfrak{p}_j$, we have $\langle [Z,X],Y\rangle =-\langle X, [Z,Y]\rangle =0$,
since $[\mathfrak{p}_i, \mathfrak{p}_j]=0$ and $\langle \cdot, \cdot \rangle$ is $\ad(\mathfrak{g})$-invariant.
Hence, $\langle [\mathfrak{p}_i,\mathfrak{p}_i], \mathfrak{p}_j \rangle=0$,
that implies ${\operatorname{pr}}_{\mathfrak{p}}\bigl([\mathfrak{p}_i,\mathfrak{p}_i]\bigr) \subset \mathfrak{p}_i$
and, therefore,
$\mathfrak{g}_i=\mathfrak{p}_i+[\mathfrak{p}_i,\mathfrak{p}_i]$.

Since $[\mathfrak{p}_i,\mathfrak{p}_j] =0$, then by the Jacobi equality we get
$[\mathfrak{p}_j, [\mathfrak{p}_i,\mathfrak{p}_i]] =0$.
Therefore,
$\langle [\mathfrak{p}_j, [\mathfrak{p}_i,\mathfrak{p}_i]], \mathfrak{g} \rangle =0$ and
$\langle [\mathfrak{p}_i,\mathfrak{p}_i], [\mathfrak{p}_j, \mathfrak{g}] \rangle =0$ by the $\ad(\mathfrak{g})$-invariance of
$\langle \cdot, \cdot \rangle$.
From $\langle [\mathfrak{p}_i,\mathfrak{p}_i], [\mathfrak{p}_j, \mathfrak{g}] \rangle =0$ we get
$\langle [\mathfrak{p}_i,\mathfrak{p}_i], [\mathfrak{p}_j, \mathfrak{p}_j] \rangle =0$ in particular. This equality together with
the equalities $\langle [\mathfrak{p}_i,\mathfrak{p}_i], \mathfrak{p}_j \rangle=0$,
$\langle [\mathfrak{p}_j,\mathfrak{p}_j], \mathfrak{p}_i \rangle=0$, and $\langle \mathfrak{p}_i, \mathfrak{p}_j\rangle =0$
imply
$\langle \mathfrak{g}_i, \mathfrak{g}_j\rangle =0$, since $\mathfrak{g}_i=\mathfrak{p}_i + [\mathfrak{p}_i,\mathfrak{p}_i]$
and $\mathfrak{g}_j=\mathfrak{p}_j + [\mathfrak{p}_j,\mathfrak{p}_j]$.

Since $[\mathfrak{p}_j, [\mathfrak{p}_i,\mathfrak{p}_i]] =0$  and $[\mathfrak{p}_j,\mathfrak{p}_i]=0$ for $i\neq j$, then
$[\mathfrak{p}_j, \mathfrak{g}_i] =0$,
$0=[\mathfrak{p}_j,[\mathfrak{g}_i,\mathfrak{p}_j]]=[\mathfrak{g}_i, [\mathfrak{p}_j,\mathfrak{p}_j]]$, and $[\mathfrak{g}_i, \mathfrak{g}_j] =0$.

It is clear that $\mathfrak{u} \subset \mathfrak{h}$.
Since $\langle \mathfrak{u}, [\mathfrak{p}_i,\mathfrak{p}_i]\rangle =0$ for any $i$, then
$\langle [\mathfrak{p}_i,\mathfrak{u}], \mathfrak{p}_i \rangle =0$, which means, that $[\mathfrak{u},\mathfrak{p}]=0$. It is easy to check
that $\mathfrak{u}=\{Z\in \mathfrak{h}\, |\, [Z,\mathfrak{p}]=0\}$. By the Lemma \ref{vslem1} we get that $\mathfrak{u}$ is an ideal
in $\mathfrak{g}$. Now, all assertions of the lemma are clear.

At first, we prove that $\mathfrak{g}_i=\mathfrak{p}_i+[\mathfrak{p}_i,\mathfrak{p}_i]$ for every $1\leq i \leq k$.
Consider any $j \neq i$, $1\leq j \leq k$.
Since $[\mathfrak{p}_i,\mathfrak{p}_j] =0$, then by the Jacobi equality we get
$[\mathfrak{p}_j, [\mathfrak{p}_i,\mathfrak{p}_i]] =0$.
Therefore,
$\langle [\mathfrak{p}_j, [\mathfrak{p}_i,\mathfrak{p}_i]], \mathfrak{g} \rangle =0$ and
$\langle [\mathfrak{p}_i,\mathfrak{p}_i], [\mathfrak{p}_j, \mathfrak{g}] \rangle =0$ by the $\ad(\mathfrak{h})$-invariance of
$\langle \cdot, \cdot \rangle$.
Since $\mathfrak{p}_j$ is $\ad(\mathfrak{h})$-irreducible, we get
$\langle [\mathfrak{p}_i,\mathfrak{p}_i], [\mathfrak{p}_j, \mathfrak{h}] \rangle =
\langle [\mathfrak{p}_i,\mathfrak{p}_i], \mathfrak{p}_j \rangle=0$. This means that
$\operatorname{pr}_{\mathfrak{p}}([\mathfrak{p}_i,\mathfrak{p}_i]) \subset \mathfrak{p}_i$ and, therefore,
$\mathfrak{g}_i=\mathfrak{p}_i+[\mathfrak{p}_i,\mathfrak{p}_i]$.

From $\langle [\mathfrak{p}_i,\mathfrak{p}_i], [\mathfrak{p}_j, \mathfrak{g}] \rangle =0$ we get also that
$\langle [\mathfrak{p}_i,\mathfrak{p}_i], [\mathfrak{p}_j, \mathfrak{p}_j] \rangle =0$. This equality together with
$\langle [\mathfrak{p}_i,\mathfrak{p}_i], \mathfrak{p}_j \rangle=0$ imply
$\langle [\mathfrak{p}_i,\mathfrak{p}_i], \mathfrak{g}_j \rangle =0$.

Further, $\langle [\mathfrak{p}_i,\mathfrak{p}_j], \mathfrak{g}\rangle =0$, then by the $\ad(\mathfrak{h})$-invariance of
$\langle \cdot, \cdot \rangle$ we get $\langle \mathfrak{p}_i, [\mathfrak{g},\mathfrak{p}_j]\rangle =0$, in particular,
$\langle \mathfrak{p}_i, \mathfrak{g}_j\rangle =0$.
This equation and $\langle [\mathfrak{p}_i,\mathfrak{p}_i], \mathfrak{g}_j \rangle =0$ imply
$\langle \mathfrak{g}_i, \mathfrak{g}_j\rangle =0$ for $i\neq j$, since
$\mathfrak{g}_i=\mathfrak{p}_i + [\mathfrak{p}_i,\mathfrak{p}_i]$.

Since $[\mathfrak{p}_j, [\mathfrak{p}_i,\mathfrak{p}_i]] =0$ for $i\neq j$, then
$[\mathfrak{p}_j, \mathfrak{g}_i] =0$,
$0=[\mathfrak{p}_j,[\mathfrak{g}_i,\mathfrak{p}_j]]=[\mathfrak{g}_i, [\mathfrak{p}_j,\mathfrak{p}_j]]$ and $[\mathfrak{g}_i, \mathfrak{g}_j] =0$.

It is clear that $\mathfrak{u} \subset \mathfrak{h}$.
Since $\langle \mathfrak{u}, [\mathfrak{p}_i,\mathfrak{p}_i]\rangle =0$ for any $i$, then
$\langle [\mathfrak{p}_i,\mathfrak{u}], \mathfrak{p}_i \rangle =0$, which means, that $[\mathfrak{u},\mathfrak{p}]=0$. It is easy to check
that $\mathfrak{u}=\{Z\in \mathfrak{h}\, |\, [Z,\mathfrak{p}]=0\}$. By the Lemma \ref{vslem1} we get that $\mathfrak{u}$ is an ideal
in $\mathfrak{g}$. Now, all assertions of the lemma are clear.
\end{proof}

\medskip
Let us consider any compact Lie algebra $\mathfrak{g}$ with $\ad(\mathfrak{g}$)-invariant inner product $\langle \cdot ,\cdot \rangle$.
Take any $Z\in \mathfrak{g}$ and consider the operator $L_Z:\mathfrak{g}\rightarrow \mathfrak{g}$, $L_Z(X)=[Z,X]$.
This operator is skew-symmetric, but $L_Z^2$ is a symmetric operator on $\mathfrak{g}$
with respect to $\langle \cdot, \cdot \rangle$.
Put $\mathfrak{m}_0=\kerr L_{Z} =\{Y\in \mathfrak{g}\,|\,[Z,Y]=0\}$ and
consider all non-zero eigenvalues of the operator $L_Z^2$:
$-\lambda_i^2$, $i=1,\dots,s$, where $0< \lambda_1<\lambda_1<\cdots <\lambda_s$, and the corresponding eigenspaces
$\mathfrak{m}_{\lambda_i}=\{Y\in \mathfrak{g}\,|\,L_Z^2(Y)=[Z,[Z,Y]]=-\lambda_i^2 Y\}$.
It is clear that $\mathfrak{m}_0$ is a Lie subalgebra of $\mathfrak{g}$ of maximal rank and
\begin{equation}\label{imagez}
\mathfrak{m}:=\mathfrak{m}_{\lambda_1}\oplus \mathfrak{m}_{\lambda_2} \oplus\cdots \oplus \mathfrak{m}_{\lambda_s}=\im L_Z.
\end{equation}

Note that this decomposition and the decomposition $\mathfrak{g}=\mathfrak{m}_0\oplus\mathfrak{m}$ are $\langle \cdot, \cdot \rangle$-orthogonal.
It follows from the simple observation: If $P\subset \mathfrak{g}$ is invariant subspace of the operator $L_Z^2$,
then its $\langle \cdot, \cdot \rangle$-orthogonal complement $P^\bot$ is also invariant subspace of $L_Z^2$.

We will need the following simple

\begin{lemma}\label{newle1}
For any $X\in \mathfrak{m}_0$ and $Y\in \mathfrak{m}_{\alpha}$ we have $[X,Y]\in \mathfrak{m}_{\alpha}$, i.e.
every $\mathfrak{m}_{\alpha}$ is $\ad(\mathfrak{m}_0)$-invariant. In particular, $[Z,\mathfrak{m}_{\alpha}]\subset \mathfrak{m}_{\alpha}$.
\end{lemma}

\begin{proof} We have $L_Z([X,Y])=[X,[Z,Y]]$ and $L_Z^2([X,Y])=[X,[Z,[Z,Y]]]=-\lambda_i^2[X,Y]$.
\end{proof}

Now we define a linear operator $\sigma:\mathfrak{m} \rightarrow \mathfrak{m}$ as follows:
\begin{equation}\label{soprya}
\sigma(Y)=\frac{1}{{\lambda}_i}[Z,Y],\qquad Y\in \mathfrak{m}_{\lambda_i}.
\end{equation}

In fact, this operator define a complex structure on a flag manifold $G/C_G(Z)$, where $C_G(Z)$ is a centralizer of $Z$ in the group $G$
(see e.g. chapter 8 in \cite{Bes}).
\medskip

For any $X,Y \in \mathfrak{m}$ we define also
\begin{equation}\label{plusminus}
[X,Y]^+:=\frac{1}{2}\bigl([X,Y]-[\sigma(X),\sigma(Y)] \bigr),\quad [X,Y]^-:=\frac{1}{2}\bigl([X,Y]+[\sigma(X),\sigma(Y)] \bigr).
\end{equation}
Obviously, $[X,Y]=[X,Y]^++[X,Y]^-$.

\begin{pred}\label{newpr1} In the above notation, $\sigma(\mathfrak{m}_{\alpha})\subset \mathfrak{m}_{\alpha}$ for all ${\alpha}$  and $\sigma(\sigma(Y))=-Y$
for all $Y\in \mathfrak{m}$.
If $U\in \mathfrak{m}_{\alpha}$, $V\in \mathfrak{m}_{\beta}$, then
$$
[U,V]^+\in \mathfrak{m}_{\alpha+\beta}, \quad
[U,V]^-\in \mathfrak{m}_{|\alpha-\beta|}.
$$
\end{pred}

\begin{proof} The first assertion follows from lemma \ref{newle1}, the equality $\sigma^2=-\Id$ is obvious.
By this equality and definitions we have
\begin{eqnarray*}
L_Z([U,V])=[Z,[U,V]]=[[Z,U],V]+[U,[Z,V]]={\alpha}\,[\sigma(U),V]+{\beta}[U,\sigma(V)],\\
L_Z([\sigma(U),\sigma(V)])\!=\![[Z,\sigma(U)],\sigma(V)]+[\sigma(U),[Z,\sigma(V)]]\!=\!\!-{\alpha}\,[U,\sigma(V)]-{\beta}[\sigma(U),V],\\
L_Z^2([U,V])=[Z,[Z,[U,V]]]=-({\alpha}^2+{\beta}^2)[U,V]+2{\alpha}{\beta}[\sigma(U),\sigma(V)],\\
L_Z^2([\sigma(U),\sigma(V)])=[Z,[Z,[\sigma(U),\sigma(V)]]]=-({\alpha}^2+{\beta}^2)[\sigma(U),\sigma(V)]+2{\alpha}{\beta}[U,V].
\end{eqnarray*}
Hence, we get $L_Z^2\bigl([U,V]-[\sigma(U),\sigma(V)]\bigr)=-({\alpha}+{\beta})^2\bigl([U,V]-[\sigma(U),\sigma(V)]\bigr)$ and
$L_Z^2\bigl([U,V]+[\sigma(U),\sigma(V)]\bigr)=-({\alpha}-{\beta})^2\bigl([U,V]+[\sigma(U),\sigma(V)]\bigr)$.
\end{proof}
\smallskip

Note that $\rk(\mathfrak{m}_0)=\rk(\mathfrak{g})$. This means that there is a unique (up to order of summands) $\langle \cdot, \cdot \rangle$-orthogonal
$\ad(\mathfrak{m}_0)$-invariant decomposition of $\mathfrak{m}$ into $\ad(\mathfrak{m}_0)$-irreducible summands (see e.g. theorem 5.3 in \cite{Kos57})
\begin{equation}\label{imagezn}
\mathfrak{m}=\mathfrak{p}_1\oplus \mathfrak{p}_2 \oplus\cdots \oplus \mathfrak{p}_t.
\end{equation}

Note, that every $\mathfrak{p}_i$, $1\leq i\leq t$, is a subspace of a suitable $\mathfrak{m}_{\lambda_j}$, $1\leq j \leq s$.
Note also that every $\mathfrak{p}_i$ is invariant under the operator $\sigma$, because $Z\in \mathfrak{m}_0$.

\section{Root space decomposition for simple compact Lie algebras}\label{rsdsec}

We give here some information on the root system of a compact simple
Lie algebra $(\mathfrak{g},\langle \cdot,\cdot \rangle=-B)$ with the
Killing form $B$, that can be found e.g. in \cite{Burb4, Hel, Sam, WolfBook}.

Let us fix a Cartan subalgebra  $\mathfrak{t}$ (that is maximal abelian subalgebra) of Lie algebra~$\mathfrak{g}$.
There is a set $\Delta$ (\textit{root system}) of (non-zero) real-valued linear form $\alpha \in \mathfrak{t}^{\ast}$ on the Cartan subalgebra
$\mathfrak{t}$, that are called \textit{roots}. Let us consider some positive root system $\Delta^+ \subset \Delta$. Recall that for any
$\alpha \in \Delta$ exactly one of the roots $\pm \alpha$ is positive (we denote it by $|\alpha|$).
The Lie algebra $\mathfrak{g}$ admits a direct
$\langle \cdot,\cdot \rangle$-orthogonal decomposition
\begin{equation}\label{rsd}
\mathfrak{g}=\mathfrak{t}\oplus \bigoplus_{\alpha \in \Delta^+} \mathfrak{v}_{\alpha}
\end{equation}
into vector subspaces, where
each subspace $\mathfrak{v}_{\alpha}$ is  2-dimensional and
$\ad(\mathfrak{t})$-invariant. Using the
restriction (of non-degenerate) inner product $\langle \cdot,\cdot \rangle$ to
$\mathfrak{t}$, we will naturally identify  $\alpha$ with some vector
in $\mathfrak{t}$. Note that $[\mathfrak{v}_{\alpha}, \mathfrak{v}_{\alpha}]$ is one-dimensional subalgebra in~$\mathfrak{t}$ spanned on
the root $\alpha$,
and $[\mathfrak{v}_{\alpha}, \mathfrak{v}_{\alpha}]\oplus \mathfrak{v}_{\alpha}$ is a Lie algebra isomorphic to $su(2)$.
The vector subspaces $\mathfrak{v}_{\alpha}$, $\alpha \in \Delta^+$, admit bases
$\{U_{\alpha},V_{\alpha}\}$, such that $\langle U_{\alpha},U_{\alpha}\rangle=\langle V_{\alpha},V_{\alpha} \rangle= 1$,
$\langle U_{\alpha},V_{\alpha}\rangle=0$
and
\begin{equation}\label{N}
[H,U_{\alpha}]=\langle \alpha,H \rangle V_{\alpha},\quad
[H,V_{\alpha}]=-\langle \alpha,H \rangle U_{\alpha}, \quad \forall H\in \mathfrak{t}, \quad
[U_{\alpha},V_{\alpha}]={\alpha}.
\end{equation}

Note also, that $[\mathfrak{v}_{\alpha},\mathfrak{v}_{\beta}]=\mathfrak{v}_{\alpha+\beta}+\mathfrak{v}_{|\alpha-\beta|}$,
assuming  $\mathfrak{v}_{\gamma}:=\{0\}$ for $\gamma \notin \Delta^+$.
\medskip

For a positive root system $\Delta^+$ the (closed) Weyl chamber is defined by the equality
\begin{equation}\label{carcham}
C=C(\Delta^+):=\{H\in\mathfrak{t}\,|\,\langle \alpha,H\rangle \geq 0\,\, \forall \alpha\in \Delta^+\}.
\end{equation}
Recall some important properties of the Weyl group $W=W(\mathfrak{t})$ of the Lie algebra $\mathfrak{g}$, that acts on the Cartan subalgebra
$\mathfrak{t}$.

(i) For every root $\alpha \in \Delta \subset \mathfrak{t}$ the Weyl
group $W$ contains the orthogonal reflection $\varphi_{\alpha}$ in
the plane $P_{\alpha}$, which is orthogonal to the root $\alpha$
with respect to $\langle \cdot,\cdot \rangle$. It is easy to see that $\varphi_{\alpha}(H)=H-2\frac{\langle H,\alpha\rangle}
{\langle \alpha,\alpha\rangle}\alpha$, $H\in \mathfrak{t}$.

(ii) Reflections from (i) generate $W$.

(iii) The root system $\Delta$ is invariant under the action of the Weyl group $W$.

(iv) $W$ acts irreducible on $\mathfrak{t}$ and simply transitively on the set of positive root systems.
For any $H\in \mathfrak{t}$,there is $w\in W$, such that $w(H)\in  C(\Delta^+)$.

(v)  For any $X\in \mathfrak{g}$, there is an inner automorphism $\psi$ of $\mathfrak{g}$ such that $\psi(X)\in \mathfrak{t}$.
For any $w\in W$, there is an inner automorphism $\eta$ of $\mathfrak{g}$, such that $\mathfrak{t}$ is stable under $\eta$, and the restriction of $\eta$
to $\mathfrak{t}$ coincides with $w$.

(vi) The Weyl group $W$ acts transitively on the set of positive roots of  fixed length.
\smallskip

For a given positive root system $\Delta^+$, we denote by $\Pi=\{\pi_1,\dots,\pi_r\}$ the set of simple root in $\Delta^+$, $r=\rk(\mathfrak{g})$.
Note that every $\alpha\in \Delta^+$ there is a unique decomposition $\alpha=\sum_{i=1}^r a_i \pi_i$ with non-negative integer $a_i$.
There is the maximal root $\beta\in \Delta^+$ that is characterized by the fact that $\beta -\alpha$ is a non-negative linear combination of roots
in $\Delta^+$ for all other positive roots  $\alpha$.
We will denote this root by $\alpha_{\max}$.

\smallskip

We list below some important properties of some root systems which we will use later (see \cite{Bourbaki}).
In all cases $(\mathfrak{t},\langle \cdot,\cdot\rangle)$ is identified with a suitable subspace of Euclidean space $\mathbb{R}^n$ with the orthonormal
basis $\{e_i\}$.
\smallskip

{\bf Case 1.} $A_l=su(l+1)$, $l\geq 1$: $\mathfrak{t}=\!\{(x_1,x_2,\dots,x_{l+1})\in \mathbb{R}^{l+1}\,|\,x_1+x_2+\cdots +x_{l+1}=0\}$,
\begin{eqnarray*}
\Delta=\{ e_i-e_j\,|\, i\neq j, 1\leq i,j \leq l+1\},\\
\Pi=\{\pi_i=e_i-e_{i+1}\,|\,i=1,\dots,l\}, \quad \alpha_{\max}=\sum_{i=1}^l \pi_i.
\end{eqnarray*}

{\bf Case 2.} $B_l=so(2l+1)$, $l\geq 2$: $\mathfrak{t}=\mathbb{R}^l$,
\begin{eqnarray*}
\Delta=\{\pm e_i\,|\, i=1,2,\dots,l \}\cup\{ \pm e_i\pm e_j\,|\,1\leq i<j \leq l\},\\
\Pi=\{\pi_i=e_i-e_{i+1}\,|\,i=1,\dots,l-1\}\cup\{\pi_l=e_l\}, \quad \alpha_{\max}=\pi_1+2\sum_{i=2}^l \pi_i.
\end{eqnarray*}

{\bf Case 3.} $C_l=sp(l)$, $l\geq 2$: $\mathfrak{t}=\mathbb{R}^l$,
\begin{eqnarray*}
\Delta=\{\pm 2e_i\,|\, i=1,2,\dots,l \}\cup\{ \pm e_i\pm e_j\,|\,1\leq i<j \leq l\},\\
\Pi=\{\pi_i=e_i-e_{i+1}\,|\,i=1,\dots,l-1\}\cup\{\pi_l=2e_l\},\quad \alpha_{\max}=2\sum_{i=1}^{l-1} \pi_i+\pi_l.
\end{eqnarray*}

{\bf Case 4.} $D_l=so(2l)$, $l\geq 3$: $\mathfrak{t}=\mathbb{R}^l$,
\begin{eqnarray*}
\Delta=\{ \pm e_i\pm e_j\,|\,1\leq i<j \leq l\},\\
\Pi=\{\pi_i=e_i-e_{i+1}\,|\,i=1,\dots,l-1\}\cup\{\pi_l=e_{l-1}+e_l\},\\
\alpha_{\max}=\pi_1+2\sum_{i=2}^{l-2} \pi_i+\pi_{l-1}+\pi_l.
\end{eqnarray*}

{\bf Case 5.} $e_6$: $\mathfrak{t}=\{(x_1,x_2,\dots,x_8)\in \mathbb{R}^8\,|\, x_6=x_7=-x_8\}$,
{\small\begin{eqnarray*}
\Delta=\{ \pm e_i\pm e_j\,|\,1\leq i<j \leq 5\}\cup\\
\left\{\pm \frac{1}{2}\bigl(e_8-e_7-e_6+\sum_{i=1}^5 (-1)^{v_i} e_i\bigr)\,|\, v_i\in\{0,1\},\,\sum_{i=1}^5 v_i \mbox{ is even} \right\},\\
\Pi=\left\{\pi_1=\frac{1}{2}(e_1+e_8)-\frac{1}{2}\sum_{i=2}^7 e_i \right\}\cup\{\pi_2=e_1+e_2\}\cup\{\pi_i=e_{i-1}-e_{i-2}\,|\,i=3,\dots,6\},\\
\alpha_{\max}=\pi_1+2\pi_2+2\pi_3+3\pi_4+2\pi_5+\pi_6.
\end{eqnarray*}}

{\bf Case 6.} $e_7$: $\mathfrak{t}=\{x\in \mathbb{R}^8\,|\, x \mbox{ orthogonal to } e_7+e_8\}$,
{\small \begin{eqnarray*}
\Delta=\{ \pm e_i\pm e_j\,|\,1\leq i<j \leq 6\}\cup\{\pm (e_7-e_8)\}\cup\\
\left\{\pm \frac{1}{2}\bigl(e_7-e_8+\sum_{i=1}^6 (-1)^{v_i} e_i\bigr)\,|\, v_i\in\{0,1\},\,\sum_{i=1}^6 v_i \mbox{ is odd} \right\},\\
\Pi=\left\{\pi_1=\frac{1}{2}(e_1+e_8)-\frac{1}{2}\sum_{i=2}^7 e_i \right\}\cup\{\pi_2=e_1+e_2\}\cup\{\pi_i=e_{i-1}-e_{i-2}\,|\,i=3,\dots,7\},\\
\alpha_{\max}=2\pi_1+2\pi_2+3\pi_3+4\pi_4+3\pi_5+2\pi_6+\pi_7.
\end{eqnarray*}}
\medskip

Recall that a simple root $\pi_i$ is called {\it non-compact} if, for every $\alpha \in \Delta$ either $\alpha$ is of the form
$\alpha=\pm \sum_{j\neq i} a_j \pi_j$, or $\alpha$ is of the form $\alpha=\pm \left(\pi_i+\sum_{j\neq i} a_j \pi_j\right)$, where
$a_i \geq 0$.  The Lie algebras from {\bf cases 1-6} are exactly simple Lie algebras that have non-compact roots. These are the roots with coefficient $1$
in the decomposition of the maximal root $\alpha_{\max}$.  Non-compact root are closely related to Hermitian symmetric spaces, see
\cite{Wolf64} or \cite{WolfBook} for a comprehensive description.

\section{Killing fields of constant length on \\ compact homogeneous Riemannian manifolds}\label{gen1}

In this section we suppose that $(M,g)$ is a compact Riemannian manifold,
and $G$ is a connected compact transitive isometry group
of $(M,g)$. We fix any $\Ad(G)$-invariant inner product $\langle \cdot, \cdot\rangle$ on the Lie algebra $\mathfrak{g}$.

\begin{pred}\label{pre1}
For a Killing field $Z\in \mathfrak{g}$, consider $A=\kerr L_Z$ and $B= \im L_Z$,  where
$L_Z:\mathfrak{g} \rightarrow \mathfrak{g}$, $L_Z(X)=[Z,X]$. Then the following conditions are equivalent:

{\rm 1)} $Z$ has constant length on $(M,g)$;

{\rm 2)} $g([Z,[Z,Y]],Z)=0$ on $(M,g)$ for every $Y\in \mathfrak{g}$;

{\rm 3)} $g(Z,Y)=0$  on $(M,g)$ for any $Y \in B$.
\end{pred}

\begin{proof}
If $Z$ is of constant length, then for any $X \in \mathfrak{g}$ we have
$g([Z,X],Z)=0$ at every point of $M$ by lemma \ref{le1}. In particular we may choose $X=[Z,Y]$.
Therefore,~1) implies 2).

Suppose that $g([Z,[Z,Y]],Z)=0$ for every $Y\in \mathfrak{g}$.
If we consider any $\Ad(G)$-invariant inner product $\langle \cdot,\cdot \rangle$ on
$\mathfrak{g}$, then the operator $L_Z:\mathfrak{g} \rightarrow \mathfrak{g}$ is skew-symmetric.
Hence the operator $L_Z^2$ is symmetric with non-positive eigenvalues.
We have $\mathfrak{g}=A\oplus B$ and consider
$B=B_1\oplus \cdots \oplus B_k$, where $B_i$ are the eigenvalue spaces for the operator $L^2_Z$ with
the eigenvalues $-{\lambda}_i^2 \neq 0$. For any $Y \in B_i$ we have
$$
0=g([Z,[Z,Y]],Z)=g(L_Z^2(Y),Z)=-{\lambda}_i^2 g(Y,Z).
$$
Therefore, $g(Z,B_i)$=0 for any $i=1,\dots,k$, and 2) implies 3).

Now suppose that $g(Z,Y)=0$ for any $Y \in B$.  Then $0=Z\cdot g(Y,Z)=-g([Y,Z],Z)$. Moreover, we obviously have $g([Y,Z],Z)=0$  for any $Y\in A$.
By lemma \ref{le1} we get that $Z$ has constant length.
\end{proof}

Note that similar result are known in the literature, see e.g. lemma 2.4 in \cite{BMS2003}.

\medskip

\begin{proof}[Proof of theorem \ref{theorem1}]
It is easy to see that $\mathfrak{g}_i=A_i\oplus B_i$, where
$A_i=\kerr L_{Z_i} \cap \mathfrak{g}_i=\kerr L_{Z} \cap \mathfrak{g}_i$ and
$B_i=\im L_{Z_i}=\im L_{Z} \cap \mathfrak{g}_i$. Note, that the operators $L_Z$ and $L_{Z_0+Z_i}$
are invertible on $B_i$.

Since $\mathfrak{g}_i$ is simple, then the pair $(\mathfrak{g}_i,A_i)$ is effective
(this means that the subalgebra $A_i$ contains no nontrivial ideal of $\mathfrak{g}_i$).
By lemma \ref{vslem1}, $B_i+[B_i,B_i]$ is an ideal in $\mathfrak{g}_i$, hence
$B_i+[B_i,B_i]=\mathfrak{g}_i$.

Since the operators $L_Z$ invertible on $B_i$,  we get $g(B_i,Z)=0$ for any $1\leq i \leq l$ by proposition \ref{pre1}.
Take any $j\neq i$, $1\leq j \leq l$. Then
$$
0=B_j\cdot g(B_i,Z)=g(B_i,[B_j,Z]),
$$
since $[B_i,B_j]=0$. Then we get that $g(B_i,B_j)=0$ (we have used that the operator $L_Z$ is invertible on $B_j$).
Further,
$0=B_i\cdot g(B_i,B_j)=g([B_i,B_i],B_j)$. Therefore, $g(\mathfrak{g}_i,B_j)=0$, and
$0=B_j\cdot g(\mathfrak{g}_i,B_j)=g(\mathfrak{g}_i,[B_j,B_j])$. Hence,
$g(\mathfrak{g}_i,\mathfrak{g}_j)=0$, that proves the first assertion.

Since $Z$ has constant length, then by proposition \ref{pre1}, we get
$g([Z,[Z,Y]],Z)=0$ for every $Y \in B_i$. But in this case $L_Z^2 (Y)=L_{Z_0+Z_i}^2(Y)$ and
$g(\mathfrak{g}_i,Z_j)=0$ for any $j\neq i$. Therefore, we get
$g\bigl([Z_0+Z_i,[Z_0+Z_i,Y]],Z_0+Z_i\bigr)=0$ for any $Y\in B_i$. It is obvious that the same is true for
$Y\in A_i$. It is clear also that $g([Z_0+Z_i,[Z_0+Z_i,Y]],Z_0+Z_i)=0$ for any $Y\in \mathfrak{c}$ and for any $Y\in \mathfrak{g}_j$, $j\neq i$.
Therefore, $Z_0+Z_i$ has constant length by proposition \ref{pre1}.

Let us prove the last assertion of the theorem. Since  $Z_0+Z_i$,
$1\leq i \leq l$, has constant length, then $g\bigl([Z_0+Z_i,[Z_0+Z_i,Y]],Z_0+Z_i\bigr)=0$ for any $Y\in B_i$ by proposition \ref{pre1}.
On the other hand, $g(\mathfrak{g}_i, \mathfrak{g}_j)=0$ for every
$i\neq j$, $1\leq i,j \leq l$. Therefore, $g([Z,[Z,Y]],Z)=g([Z_0+Z_i,[Z_0+Z_i,Y]],Z)=0$ for any $Y\in B_i$.
This equality holds for every $1\leq i \leq l$. It is easy to see that
$g([Z,[Z,Y]],Z)=0$ for any $Y\in \mathfrak{c}$ and for any $Y\in \mathfrak{g}_j$, $j >l$.
Therefore, $Z$ has constant length by proposition \ref{pre1}.
\end{proof}

\medskip

\begin{proof}[Proof of corollary \ref{cortheo1}] For $X\in \mathfrak{g}_i$ and $Y\in \mathfrak{g}_j$, where $i\neq j$, we have $g(X,Y)=0$
by theorem \ref{theorem1}. Then by lemma \ref{lemma3}, we have $\nabla_XY=0$, because $[X,Y]=0$. Moreover, for any $Z\in \mathfrak{g}_i$
we get $0=Z\cdot g(X,Y)=g(\nabla_ZX,Y)+g(X,\nabla_ZY)=g(\nabla_ZX,Y)$. Therefore, $\nabla_{\mathfrak{g}_i}\mathfrak{g}_i \subset \mathfrak{g}_i$ and
$\mathfrak{g}_i$ is a parallel distribution. Since $i=1,\dots,k$ is an arbitrary, we get a global splitting
by the de Rham theorem (see e.g. theorem 6.1 of chapter IV in \cite{KN}) for simply connected $M$.
\end{proof}

\bigskip

The rest of this section is devoted to the proof of theorems \ref{theorem2} and \ref{theorem3}.
Let us consider a Killing field $Z\in \mathfrak{g}$ of constant length.
Recall that the operators $L_Z$ and $L_Z^2$ are skew-symmetric and symmetric respectively.

We consider $\mathfrak{m}_0=\kerr L_{Z} =\{Y\in \mathfrak{g}\,|\,[Z,Y]=0\}$ and
the eigenspace decomposition
\begin{eqnarray*}
\mathfrak{m}&=&\mathfrak{m}_{\lambda_1}\oplus \mathfrak{m}_{\lambda_2} \oplus\cdots \oplus \mathfrak{m}_{\lambda_s}=\im L_Z, \\
\mathfrak{m}_{\lambda_i}&=&\{Y\in \mathfrak{g}\,|\,L_Z^2(Y)=[Z,[Z,Y]]=-\lambda_i^2 Y\} .
\end{eqnarray*}

Since for any $Y \in \mathfrak{m}_{\lambda_i}$ we have $L_Z^2(Y)=[Z,[Z,Y]]=-\lambda_i^2 Y$, then
$\lambda_i^2 g(Z,Y)=-g([Z,[Z,Y]],Z)=0$ by proposition \ref{pre1}. Therefore, we get

\begin{corollary}\label{orthog}
The equality $g(Z,\mathfrak{m})=0$ holds on $(M,g)$.
\end{corollary}

\begin{pred}\label{newpr2}
Suppose that $\mathfrak{p}$ is an $\ad(\mathfrak{m}_0)$-invariant submodule in $\mathfrak{m}$,
and $\mathfrak{p}^\bot$ is its $\langle \cdot,\cdot \rangle$-orthogonal complement in $\mathfrak{m}$.
Then $g(\mathfrak{p},\mathfrak{p}^\bot)=0$ on $(M,g)$.
In particular, $g(\mathfrak{m}_{\alpha},\mathfrak{m}_{\beta})=0$ on $(M,g)$ for $\alpha \neq \beta$.
\end{pred}

\begin{proof}
By corollary \ref{orthog} we have $g(Z,\mathfrak{m})=0$. Let us consider any $X\in \mathfrak{p}$ and
any $Y \in \mathfrak{p}^\bot$.
Then $0=Y\cdot g(Z,X)=g([Y,Z],X)+g(Z,[Y,X])$.  Note that
$\langle [Y,X],W\rangle=\langle Y,[X,W]\rangle=0$ for any $W\in \mathfrak{m}_0$, hence $[Y,X]\in \mathfrak{m}$.
From this  and proposition \ref{pre1} we get $g(Z,[Y,X])=0$. Therefore,
$g([Y,Z],X)=-g(L_Z(Y),X)=0$. Since $\mathfrak{p}$ is an invariant subspace of $L_Z$, and $L_Z$ is invertible on $\mathfrak{m}$,
this proves the proposition.
\end{proof}
\medskip

Now we consider decomposition (\ref{imagezn}) in our case:
$\mathfrak{m}=\mathfrak{p}_1\oplus \mathfrak{p}_2 \oplus\cdots \oplus \mathfrak{p}_t$.
From proposition \ref{newpr2} we obviously get

\begin{corollary}\label{decomort}
The equality $g(\mathfrak{p}_i,\mathfrak{p}_j)=0$ holds on $(M,g)$ for every $i\neq j$.
\end{corollary}

We will need the following construction: For all $i\geq 1$ and all $X\in \mathfrak{p}_i$ we define
\begin{equation}\label{hvector}
h(X)=[X,\sigma(X)],
\end{equation}
where $\sigma(X)$ is defined by  the equality (\ref{soprya}): ${\lambda}_i \sigma(Y)=[Z,Y]$.

\begin{pred}\label{newpr3}
For all $i\geq 1$ and all $X\in \mathfrak{p}_i$ we have $h(X)\in \mathfrak{m}_0$ and $g(X,\sigma(X))=g(X,h(X))=g(\sigma(X),h(X))=0$ on $(M,g)$.
Moreover, for any $j\neq i$, $j\geq 1$, the equality $g(h(U),\mathfrak{p}_j)=0$ holds on $(M,g)$.
\end{pred}

\begin{proof}
It is easy to see that $[X,\sigma(X)]^+=0$  (see (\ref{plusminus})) and
$h(X)=[X,\sigma(X)]=[X,\sigma(X)]^- \in \mathfrak{m}_0$ for every $X\in \mathfrak{p}_i$ and $i\geq 1$
by proposition \ref{newpr1}. Further, we have $g(Z,X)=g(Z,\sigma(X))=0$ by corollary \ref{orthog}. Hence
\begin{eqnarray*}
0=X\cdot g(Z,X)=g([X,Z],X)=-\alpha\, g(\sigma(X),X),\,\mbox{ where }\alpha \mbox{ is taken from }\,\mathfrak{p}_i\subset \mathfrak{m}_{\alpha},\\
0=X\cdot g(X,\sigma(X))=g(X,h(X)),\,\mbox{ and }\,0=\sigma(X)\cdot g(X,\sigma(X))=-g(h(X),\sigma(X)).
\end{eqnarray*}

Now, we prove the last assertion. Take any $Y\in \mathfrak{p}_j\subset \mathfrak{m}_{\beta}$. By corollary~\ref{decomort},
we have $g(\sigma(X),Y)=g(\sigma(X),\sigma(Y))=0$, therefore,
\begin{eqnarray*}
0&=&X\cdot g(\sigma(X),Y)=g(h(X),Y)+g(\sigma(X),[X,Y]),\\
0&=&\sigma(X)\cdot g(\sigma(X),\sigma(Y))=g(\sigma(X),[\sigma(X),\sigma(Y)]).
\end{eqnarray*}
From this we get
$$
-g(h(X),Y)=g(\sigma(X),[X,Y]\pm [\sigma(X),\sigma(Y)])=g\bigl(\sigma(X),[X,Y]^{\pm}\bigr).
$$
By proposition \ref{newpr1},
we have $[X,Y]^+\in \mathfrak{m}_{\alpha+\beta}$ and $[X,Y]^-\in\mathfrak{m}_{|\alpha-\beta|}$. Since $\alpha+\beta>\alpha$,
we obtain $g(\sigma(X),[X,Y]^{+})=0$ by proposition \ref{newpr2}.
Therefore, $g(h(X),Y)=0$.
\end{proof}
\medskip

\begin{pred}\label{newpr3.5}
Suppose that $i\geq 1$ and $W\in \mathfrak{m}_0$ are such that $\langle W, h(X)\rangle=0$ for all  $X\in \mathfrak{p}_i$.
Then $[W,\mathfrak{p}_i]=0$.
\end{pred}

\begin{proof}
Consider the operators $P,Q:\mathfrak{p}_i \rightarrow \mathfrak{p}_i$, such that
$P(X)=[W,X]$ and $Q(X)=[Z,X]$. These operators are skew-symmetric with respect to $\langle\cdot,\cdot \rangle$ and commute one with other (since $[Z,W]=0$).
Moreover, $Q^2=-\lambda_k^2 \Id$, where $\mathfrak{p}_i\subset \mathfrak{m}_{\lambda_k}$.
Therefore, the operator $R:=PQ(=QP)$ is symmetric and $R^2=PQ^2P=-\lambda_k^2 P^2$. Now, for any $U\in \mathfrak{p}_i$ we get
\begin{eqnarray*}
0=-\lambda_k\langle W, h(U) \rangle=-\langle W, [U,[Z,U]] \rangle=\langle [U,W],[Z,U]] \rangle=\\
-\langle [W,U],[Z,U]] \rangle=\langle [Z,[W,U]],U \rangle=\langle R(U),U \rangle.
\end{eqnarray*}
Therefore, $R=0$ and $P=0$, i.e. $[W,\mathfrak{p}_i]=0$.
\end{proof}
\medskip

\begin{pred}\label{newpr4}
Let $Z =Z_0+Z_1\in \mathfrak{g}$ be a Killing field of constant length on $(M,g)$, where
$Z_0\in \mathfrak{c}$, $Z_1 \in \mathfrak{g}_1$, $Z_1 \neq 0$, and let $\mathfrak{k}$ be the centralizer of $Z$ {\rm(}and $Z_1${\rm)} in $\mathfrak{g}_1$.
Then one of the following  statements holds:

{\rm 1)} The equality $g(\mathfrak{k},\mathfrak{m})=0$ is fulfilled on $(M,g)$;

{\rm 2)} The Lie algebra pair $(\mathfrak{g}_1, \mathfrak{k})$ is irreducible Hermitian symmetric and
the center of~$\mathfrak{k}$ is a one-dimensional Lie algebra spanned by the vector~$Z_1$.
\end{pred}

\begin{proof}
We apply all above constructions of this section to prove this key proposition. Note that
\begin{equation}\label{decomtwo}
\mathfrak{m}=\mathfrak{p}_1\oplus\cdots\oplus\mathfrak{p}_t\subset \mathfrak{g}_1\quad \mbox{and} \quad
\mathfrak{m}_0=\mathfrak{c}\oplus\mathfrak{k}\oplus\mathfrak{g}_2\oplus\mathfrak{g}_3\oplus \cdots \oplus\mathfrak{g}_k
\end{equation}
in our case.
\medskip

First, let us suppose that $[\mathfrak{p}_i,\mathfrak{p}_j]=0$ for all $i\neq j$, $1\leq i,j \leq s$.
As we have noted, the first decomposition in (\ref{decomtwo}) and the decomposition $\mathfrak{g}_1=\mathfrak{k}\oplus \mathfrak{m}$
are orthogonal with respect to  $\langle\cdot,\cdot \rangle$.
Now, by lemma \ref{vslem2} we get that every $\mathfrak{p}_i+[\mathfrak{p}_i,\mathfrak{p}_i]$, $1\leq i \leq s$, is an ideal in $\mathfrak{g}_1$.
Since $\mathfrak{g}_1$ is a simple Lie algebra, we get that $s=1$, and $\mathfrak{p}_1=\mathfrak{m}_{\lambda_1}$.

For $X,Y \in \mathfrak{p}_1$, proposition \ref{newpr1} implies
$[X,Y]^+=1/2\bigl([U,V]-[\sigma(U),\sigma(V)] \bigr)\in \mathfrak{m}_{2\lambda_1}$, hence $[X,Y]^+=0$ (recall that $s=1$ and $\lambda_1<2\lambda_1$) and
$[X,Y]=[X,Y]^-=1/2\bigl([X,Y]+[\sigma(X),\sigma(Y)] \bigr)\in \mathfrak{m}_{0}$. This means that $[\mathfrak{p}_1,\mathfrak{p}_1]\subset\mathfrak{k}$.
Therefore, the Lie algebra pair $(\mathfrak{g}_1, \mathfrak{k})$ is irreducible symmetric.
But $\mathfrak{k}$ is the centralizer of the vector $Z_1$ in $\mathfrak{g}_1$, therefore, it has non-zero center. Consequently,
$(\mathfrak{g}_1, \mathfrak{k})$ is an Hermitian irreducible symmetric pair (see e.g. proposition 8.7.12 in \cite{WolfBook}).
\medskip

In the rest of the proof we will assume that there are  indices $i\neq j$, $1\leq i,j \leq s$, such that $[\mathfrak{p}_i,\mathfrak{p}_j]\neq 0$.
\medskip

Let $S$ be a linear span of all vectors $h(U)$ (see (\ref{hvector})), where $U\in \mathfrak{p}_i$, $i \geq 1$.
In our case $\mathfrak{p}_i\subset \mathfrak{g}_1$, hence $h(U)\in \mathfrak{g}_1$ and $S \subset \mathfrak{g}_1$.
We prove that the equality $g(S,[\mathfrak{p}_i,\mathfrak{p}_j])=0$ holds on $(M,g)$ for all $i,j\geq 1$, $i\neq j$.

Take any $h(U)$, where $U\in \mathfrak{p}_k$, $k \geq 1$. Without loss of generality we may suppose that $k\neq j$.
Then $g(h(U),\mathfrak{p}_j)=0$ by proposition \ref{newpr3}. Now, for every $X\in \mathfrak{p}_i$ and $Y\in \mathfrak{p}_j$ we have
$0=X\cdot g(h(U),Y)=g([X,h(U)],Y)+g(h(U), [X,Y])$. Since $[h(U),X]\in [\mathfrak{m}_0, \mathfrak{p}_i]\subset \mathfrak{p}_i$ we have
$g([X,h(U)],Y)=0$ by corollary \ref{decomort}. Hence, $g(h(U), [X,Y])=0$ that implies $g(h(U),[\mathfrak{p}_i,\mathfrak{p}_j])=0$ and
$g(S,[\mathfrak{p}_i,\mathfrak{p}_j])=0$.
\medskip

Let $S^{\bot}$ be the orthogonal complement to the linear space $S$ in $\mathfrak{m}_0$
with respect to $\langle\cdot,\cdot \rangle$.  By proposition \ref{newpr3.5}  we have  $[W,\mathfrak{p}_i]=0$ for all $W \in S^{\bot}$ and all $i$,
i.e. $[S^{\bot},\mathfrak{m}]=0$.
Let $q_1$ be a maximal (by inclusion) linear subspace in  $\mathfrak{m}_0$ with the property $[q_1,\mathfrak{m}]=0$.
We get that $q_1$ is an ideal in $\mathfrak{g}$ by lemma \ref{vslem1}.

Note also that $\mathfrak{c}\oplus\mathfrak{g}_2\oplus\mathfrak{g}_3\oplus \cdots \oplus\mathfrak{g}_k\subset S^{\bot}\subset q_1$. Since the Lie algebra
$\mathfrak{g}_1$ is simple we get
$q_1=S^{\bot}=\mathfrak{c}\oplus\mathfrak{g}_2\oplus\mathfrak{g}_3\oplus \cdots \oplus\mathfrak{g}_k$, therefore, $S=\mathfrak{k}$.
Consequently, we get that $g(\mathfrak{k},[\mathfrak{p}_i,\mathfrak{p}_j])=0$ for every $i\neq j$, $1\leq i,j \leq t$.
\medskip

Since
$[\mathfrak{m}_0,[\mathfrak{p}_i,\mathfrak{p}_j]]\subset[[\mathfrak{m}_0,\mathfrak{p}_i],\mathfrak{p}_j]+
[\mathfrak{p}_i,[\mathfrak{m}_0,\mathfrak{p}_j]]\subset[\mathfrak{p}_i,\mathfrak{p}_j]$,
then any subspace of the type $[\mathfrak{p}_i,\mathfrak{p}_j]\subset \mathfrak{m}$ is $\ad(\mathfrak{m}_0)$-invariant.

Let $r_1$ be a maximal (by inclusion) $\ad(\mathfrak{m}_0)$-invariant linear subspace in  $\mathfrak{m}$ with the property
$g(\mathfrak{k},r_1)=0$. Clearly, $[\mathfrak{p}_i,\mathfrak{p}_j]\subset r_1$ for any $i\neq j$. We are going to prove that $r_1=\mathfrak{m}$.

Denote by $r_2$ the $\langle \cdot,\cdot \rangle$-orthogonal complement to $r_1$ in $\mathfrak{m}$.
Since both $r_1$ and $r_2$ are spanned by some modules $\mathfrak{p}_j$, we get that $[r_1,r_2]\subset r_1$, because
$[\mathfrak{p}_i,\mathfrak{p}_j]\subset r_1$ for any $i\neq j$. Since $\langle [r_2,r_2],r_1 \rangle = -\langle r_2, [r_2,r_1] \rangle=0$,
then $\mathfrak{k}\oplus r_2$ is a Lie subalgebra of $\mathfrak{g}_1$.
\medskip

Let $S_1$ be a linear span of all vectors $h(U)$ (see (\ref{hvector})), where $U\in \mathfrak{p}_i$ and $\mathfrak{p}_i \subset r_1$.
We know that $h(U)\in \mathfrak{k}$.

Let $S_1^{\bot}$ be an orthogonal complement to the linear space $S_1$ in $\mathfrak{k}$
with respect to $\langle\cdot,\cdot \rangle$.  By proposition \ref{newpr3.5}  we have
$[W,\mathfrak{p}_i]=0$ for all $W \in S^{\bot}_1$ and $r_1 \subset \mathfrak{p}_1$,
i.e. $[S_1^{\bot},r_1]=0$.
Let $q_2$ be a maximal (by inclusion) linear subspace in  $\mathfrak{k}\oplus r_2$ with the property $[q_2, r_1]=0$.
Therefore, $q_2$ is an ideal in $\mathfrak{g}_1$  by lemma \ref{vslem1}. Since $\mathfrak{g}_1$ is a simple Lie algebra, we get $q_2=0$.
Since $S_1^{\bot} \subset q_2$, we get $S_1=\mathfrak{k}$.

For all $i\geq 1$ and all $U\in \mathfrak{p}_i$ with $\mathfrak{p}_i \subset r_1$,
and for all $\mathfrak{p}_j\subset r_2$, we have the equality $g(\mathfrak{p}_j,h(U))=0$ on $(M,g)$ by proposition \ref{newpr3}.
Therefore, $g(\mathfrak{p}_j,S_1)=g(\mathfrak{p}_j,\mathfrak{k})=0$ for any $\mathfrak{p}_j \subset r_2$. Hence,
$g(r_2,\mathfrak{k})=0$ and $r_2=0$ by the definition of $r_1$. Consequently, we get the equality $g(\mathfrak{k},\mathfrak{m})=0$ on $(M,g)$.
\end{proof}

\bigskip

\begin{pred}\label{newpr5}
If the case {\rm 1)} of proposition \ref{newpr4} holds, then the every Killing field $Y \in \mathfrak{g}_1$ {\rm(}in particular, $Z_1${\rm)}
has constant length on $(M,g)$.
\end{pred}

\begin{proof}
According to the case {\rm 1)} of proposition \ref{newpr4}, we get that $g(\mathfrak{k}, \mathfrak{m})=0$ at every point of $M$, where
$\mathfrak{k}$ is the centralizer of $Z_1$ in $\mathfrak{g}_1$ and $\mathfrak{m}$ is
the $\ad(\mathfrak{k})$-invariant complement to $\mathfrak{k}$ in $\mathfrak{g}_1$.

Let us consider any Cartan subalgebra $\mathfrak{t}$ in $\mathfrak{g}_1$ such that $Z_1\in \mathfrak{t}$.
First, prove that every $X\in \mathfrak{t}$ has constant length on $(M,g)$.
By proposition \ref{pre1}, it suffices to prove that $g(\mathfrak{t},\mathfrak{v}_{\alpha})=0$ for all $\alpha \in \Delta^+$ in the decomposition
(\ref{rsd}), where $\mathfrak{g}:=\mathfrak{g}_1$. Indeed, take  any $X\in \mathfrak{t}$.
Note that $g([X,[X,Y]],X)=0$ for all $Y\in \mathfrak{t}$, because $[X,Y]=0$.
Moreover, $g([X,[X,Y]],X)=0$ is easily satisfied if $Y$ is from other factors of $\mathfrak{g}$.
If $Y \in \mathfrak{v}_{\alpha}$ for $\alpha \in \Delta^+$, then
$g([X,[X,Y]],X)=-\langle X,\alpha_i\rangle^2 g(X,Y)$ by (\ref{N}).
Therefore, if $g(X,\mathfrak{v}_{\alpha})=0$ for all $\alpha \in \Delta^+$, then $g([X,[X,Y]],X)=0$ for every $Y\in \mathfrak{g}$.
By proposition \ref{pre1} it implies that $X$ has constant length on $(M,g)$.

We have the equality $g(\mathfrak{t}, \mathfrak{m})=0$.
Let us consider the orbit $\{Z_1=U_1,U_2, \dots, U_l \}$ of $Z_1$
in $\mathfrak{t}$ under the action of the Weyl group $W$, see section \ref{rsdsec}.
By properties of $W$, for any  $U_i$ there is an inner automorphism $\psi_i$ of $\mathfrak{g}_1$ such that $\mathfrak{t}$ is stable under $\psi_i$
and $U_i=\psi_i(Z_1)$. Note that  we may consider $\psi_i$ as an automorphism of
$\mathfrak{g}=\mathfrak{c}\oplus\mathfrak{g}_1\oplus\cdots\oplus \mathfrak{g}_k$ (it acts identically on all other summands).
By lemma \ref{s1}, the Killing field $Z_0+U_i=Z_0+\psi_i(Z_1)=\psi_i(Z_0+Z_1)$ also has constant length.
Moreover, $g\bigl(\psi_i(\mathfrak{t}),\psi_i(\mathfrak{m})\bigr)=g\bigl(\mathfrak{t},\psi_i(\mathfrak{m})\bigr)=0$ on $(M,g)$ by the same lemma.
Now, it suffices to prove that for every $\alpha \in \Delta^+$, there is $i$ such that $\mathfrak{v}_{\alpha}\subset \psi_i(\mathfrak{m})$.
The latter is equivalent to the following: $\psi_i(Z_1)=U_i$ is not orthogonal to the root $\alpha$  with respect to $\langle\cdot,\cdot \rangle$, because
$\psi_i(\mathfrak{k})$ is the centralizer of $U_i=\psi_i(Z_1)$ in $\mathfrak{g}_1$ and $\psi_i(\mathfrak{m})$ consists of all root spaces with
$\langle \alpha,U_i\rangle\neq 0$.

Suppose that there is no suitable $\psi_i(Z_1)=U_i$, i.e. $\langle \alpha,U_i\rangle=0$ for all $i$. But $W$ acts irreducibly on $\mathfrak{t}$,
therefore $\Lin (Z_1=U_1,U_2, \dots, U_l)=\mathfrak{t}$. This contradiction proves $g(\mathfrak{t},\mathfrak{v}_{\alpha})=0$, therefore
every $X\in \mathfrak{t}$ has constant length on $(M,g)$.

Now, for an arbitrary $X\in \mathfrak{g}_1$, there is an inner automorphism $\eta$ of $\mathfrak{g}_1$ such that $\eta(X)\in \mathfrak{t}$, see
section \ref{rsdsec}. Then using lemma \ref{s1} again, we obtain that $X$ has constant length on $(M,g)$.
\end{proof}
\medskip

We will finish the proof of theorem \ref{theorem2} in the next section.
But now we are going to prove theorem \ref{theorem3}.

\begin{lemma}\label{biinl}
Let $G$ be compact Lie group {\rm(}with the Lie algebra $\mathfrak{g}${\rm)} supplied with a left-invariant Riemannian metric $\mu$.
If every $X\in \mathfrak{g}$ has constant length on the Riemannian manifold $(G,\mu)$, then $\mu$ is a bi-invariant metric.
\end{lemma}

\begin{proof}
Since for every Killing field $X,Y \in \mathfrak{g}$
we have
$$
2\mu(X,Y)=\mu(X+Y,X+Y)-\mu(X,X)-\mu(Y,Y)=\const,
$$
then for any $a \in G$
$$
\mu_e(\Ad(a)(X),\Ad(a)(Y))=\mu_a(\Ad(a)(X),\Ad(a)(Y))=\mu_e(X, Y)
$$
by lemma \ref{s1}, where $e$ is the unit in $G$.
Therefore, $\mu$ is a bi-invariant metric on $G$.
\end{proof}
\medskip

\begin{proof}[Proof of theorem \ref{theorem3}]Since $Z$ is regular, then $k=l$ in theorem \ref{theorem1} and every $Z_i$ is regular in $\mathfrak{g}_i$.
By theorem \ref{theorem1} we have $g(\mathfrak{g}_i,\mathfrak{g}_j)=0$ for every $i\neq j$, $1\leq i,j \leq k$, and all $Z_0+Z_i$ have constant length on
$(M,g)$.
Since $Z_i$ is regular in $\mathfrak{g}_i$, we get the case 1) in proposition \ref{newpr4} and, therefore,
all $X\in \mathfrak{g}_i$ has constant length on $(M,g)$ by proposition \ref{newpr5} (here we have considered $\mathfrak{g}_i$ instead of $\mathfrak{g}_1$).
Therefore, every $X\in \mathfrak{g}_{\mathfrak{s}}$,
where $\mathfrak{g}_{\mathfrak{s}}=\mathfrak{g}_1\oplus\cdots\oplus\mathfrak{g}_k$
is a semisimple part of $\mathfrak{g}$, also has constant length.
In particular, $Z_{\mathfrak{s}}:=Z-Z_0$ has constant length.
By proposition \ref{pre1} we get $g(Z,Y)=g(Z_{\mathfrak{s}},Y)=0$ for every $Y\in \mathfrak{m}$, hence,
$g(Z_0, Y)=0$ for the same $Y$. Now, for any $X\in \mathfrak{m}$ we have $0=X\cdot g(Z_0,Y)=g(Z_0,[X,Y])$. Note that
$\mathfrak{m}+[\mathfrak{m},\mathfrak{m}]$ is an ideal in $\mathfrak{g}_{\mathfrak{s}}$ by lemma \ref{vslem1},
hence, $\mathfrak{m}+[\mathfrak{m},\mathfrak{m}]=\mathfrak{g}_{\mathfrak{s}}$ and $g(Z_0,\mathfrak{g}_{\mathfrak{s}})=0$.

Now, suppose that $M$ is simply connected. In this case semisimple part of the isometry group $G$ acts transitively on $(M,g)$
(see e.g. proposition 9 on p.~94 in \cite{On1994}).
This mean that the Lie algebra $\mathfrak{g}_{\mathfrak{s}}$ generates a tangent space to $M$ at every point $x\in M$.
Since $g(Z_0,\mathfrak{g}_{\mathfrak{s}})=0$ on $M$, we get $Z_0=0$.

Further, $(M,g)$ is a direct metric product of Riemannian manifold $(M_i,g_i)$, that are corresponded to the distributions $\mathfrak{g}_i$,
by corollary \ref{cortheo1}.
Therefore, $G_i$ is an isometry group of $(M_i,g_i)$.
Since all $X\in \mathfrak{g}_i$ have constant length, the isotropy group (of any point of $(M_i,g_i)$) is discrete.
Since $M_i$ is simply connected, $G_i$ acts simply transitively on $(M_i,g_i)$, hence $(M_i,g_i)$ is isometric to
$G_i$ supplied with some left-invariant Riemannian metric $\mu_i$.
Since all $X\in \mathfrak{g}_i$ have constant length on $(G_i,\mu_i)$, then $\mu_i$ is bi-invariant
by lemma~\ref{biinl}.
\end{proof}

\section{More detailed study of the Hermitian case}\label{shcsec}

In this section we study in more details the case 2) in proposition \ref{newpr4}, which we call Hermitian.
In this case, $\mathfrak{m}$ is $\ad(\mathfrak{k})$-irreducible, in particular $L_Z^2|_{\mathfrak{m}} = -C\cdot \Id$, where $C$ is a positive constant.
Multiply $Z$ by a suitable constant we may suppose that $C=1$.
Using the root space decomposition (\ref{rsd}) for $\mathfrak{g}_1$ we get
\begin{equation}\label{rsddet}
\mathfrak{m}=\bigoplus_{\alpha \in I_1(Z)} \mathfrak{v}_{\alpha},\qquad
\mathfrak{k}=\mathfrak{t}\oplus \bigoplus_{\alpha \in I_2(Z)} \mathfrak{v}_{\alpha},
\end{equation}
where
\begin{equation}\label{i1i2}
I_1(Z)=\{\alpha \in \Delta^+\,|\, \langle Z,\alpha \rangle \neq 0\}, \qquad
I_2(Z)=\{\alpha \in \Delta^+\,|\, \langle Z,\alpha \rangle  =0\}.
\end{equation}
Without loss of generality we may assume that $Z$ is in the Cartan chamber $C(\Delta^+)$ (see (\ref{carcham})),
hence $\langle Z, \alpha\rangle >0$ for all $\alpha\in I_1(Z)$. By (\ref{N}) we have
$$
[Z,U_{\alpha}]=\langle Z,\alpha \rangle V_{\alpha},\quad
[Z,V_{\alpha}]=-\langle Z,\alpha \rangle U_{\alpha}, \quad
L_Z^2|_{\mathfrak{v}_{\alpha}}=-\langle Z,\alpha \rangle^2 \cdot \Id
$$
for all $\alpha \in I_1(Z)$.
Then, $\langle Z,\alpha \rangle=1$ for all $\alpha \in I_1(Z)$ by our assumptions.
In particular, $L_Z |_{\mathfrak{m}}=\sigma$, see (\ref{soprya}).
By $W=W(\mathfrak{t})$ we denote the corresponding Weyl group of $\mathfrak{g}_1$.

\begin{remark}\label{ncrem}
We can explain the connection of the Hermitian case with Hermitian symmetric spaces as follows.
Let $\alpha_{\max}=\sum_i a_i \pi_i$ be a maximal root for $\Delta^+$.
Then $\langle Z,\alpha_{\max}\rangle =\sum_i a_i \langle Z, \pi_i \rangle$, where all $a_i$
are integer with $a_i\geq 1$. Note that $\langle Z, \pi_i \rangle=0$ for $\pi_i\in I_2(Z)$ and
$\langle Z, \pi_i \rangle=1$ for $\pi_i\in I_1(Z)$. There is at least one $\pi_j$ in $I_1(Z)$, hence $\langle Z,\alpha_{\max}\rangle=1$,
and such $\pi_j$ should be unique, moreover, $a_j=1$.
Therefore, $\pi_j$ is a non-compact root and $\pi_i \in I_2(Z)$ for all $i\neq j$. Every non-compact root determines a Hermitian symmetric space as it was shown
in \cite{Wolf64}.
\end{remark}

\begin{lemma}\label{pr1.1}
In the above notations, for any $w\in W$ the Killing field
$Z_0+w(Z_1)$ has constant length on $(M,g)$.
\end{lemma}

\begin{proof}
By properties of $W$ (see section \ref{rsdsec}), there is an inner automorphism $\psi$ of $\mathfrak{g}_1$ such that $\mathfrak{t}$ is stable under $\psi$
and $w(Z_1)=\psi(Z_1)$. We may consider $\psi$ as an automorphism of
$\mathfrak{g}=\mathfrak{c}\oplus\mathfrak{g}_1\oplus\cdots\oplus \mathfrak{g}_k$ (it acts identically on all other summands).
By lemma \ref{s1}, the Killing field $Z_0+w(Z_1)=Z_0+\psi(Z_1)=\psi_i(Z_0+Z_1)$ also has constant length.
\end{proof}

\begin{pred}\label{pr3}
In the notation as above, for any $\alpha \in I_1(Z)$, the following equalities hold:
\begin{eqnarray*}
g(Z, \mathfrak{v}_{\alpha})=0, \quad
g(U_{\alpha},U_{\alpha})=g(V_{\alpha},V_{\alpha})=g(\alpha,Z)=\frac{1}{\langle \alpha, \alpha\rangle} \,g(\alpha,\alpha),\\
g(U_{\alpha},V_{\alpha})=g(\alpha,U_{\alpha})=g(\alpha,V_{\alpha})=0,\quad g(\alpha,\alpha)=\langle \alpha, \alpha\rangle \, g(\alpha,Z).
\end{eqnarray*}
\end{pred}

\begin{proof}
The equality $g(Z, \mathfrak{v}_{\alpha})=0$ follows from corollary \ref{orthog}.
Substituting $X:=Z$, $Z:=V_{\alpha}$ and $Y:=U_{\alpha}$ in the equality
$g([Z,[Y,X]],X)+g([Y,X],[Z,X])=0$ (see (\ref{e00}) in lemma \ref{le1}) and using $g(Z, \mathfrak{v}_{\alpha})=0$ and (\ref{N}), we get
$$
0=g([V_{\alpha},[U_{\alpha},Z]],Z)+g([[U_{\alpha},Z],[V_{\alpha},Z])=-{\langle Z, \alpha \rangle}^2g(U_{\alpha},V_{\alpha})=-g(U_{\alpha},V_{\alpha}).
$$ Further,
\begin{eqnarray*}
0=U_{\alpha}\cdot g(U_{\alpha},V_{\alpha})=g(U_{\alpha},[U_{\alpha},V_{\alpha}])=g(U_{\alpha},\alpha),\\
0=V_{\alpha}\cdot g(U_{\alpha},V_{\alpha})=g([V_{\alpha},U_{\alpha}],V_{\alpha})=-g(\alpha,V_{\alpha}).
\end{eqnarray*}
Since $g(U_{\alpha},Z)=0$, then
$$
0=V_{\alpha}\cdot g(U_{\alpha},Z)=-g(\alpha,Z)+g(U_{\alpha},{\langle Z,\alpha\rangle}U_{\alpha})=g(U_{\alpha},U_{\alpha})-g(\alpha,Z).
$$
Analogously,
$$
0=U_{\alpha}\cdot g(V_{\alpha},Z)=g(\alpha,Z)+g(V_{\alpha},-{\langle Z,\alpha\rangle}V_{\alpha})=-g(V_{\alpha},V_{\alpha})+g(\alpha,Z).
$$
Further, since $g(U_{\alpha},\alpha)=g(V_{\alpha},\alpha)=0$, we get
$$
0=U_{\alpha}\cdot g(V_{\alpha},\alpha)=g([U_{\alpha},V_{\alpha}],\alpha)+g(V_{\alpha},[U_{\alpha},\alpha])=
g(\alpha,\alpha)-\langle \alpha,\alpha \rangle g(V_{\alpha},V_{\alpha}),
$$
$$
0=V_{\alpha}\cdot g(U_{\alpha},\alpha)=g([V_{\alpha},U_{\alpha}],\alpha)+g(U_{\alpha},[V_{\alpha},\alpha])=
\langle \alpha,\alpha \rangle g(U_{\alpha},U_{\alpha})-g(\alpha,\alpha).
$$
The obtained equalities prove the proposition.
\end{proof}

\begin{lemma}\label{le2}
For every $\alpha, \beta \in I_1(Z)$ the condition $g(\mathfrak{v}_{\alpha},\mathfrak{v}_{\beta})=0$ is equivalent to
the condition $g([\mathfrak{v}_{\alpha},\mathfrak{v}_{\beta}],Z)=0$.
In particular, $g(\mathfrak{v}_{\alpha},\mathfrak{v}_{\beta})=0$ follows from $[\mathfrak{v}_{\alpha},\mathfrak{v}_{\beta}]\subset \mathfrak{m}$.

\end{lemma}

\begin{proof}
By proposition \ref{pr3} we have $g(\mathfrak{v}_{\alpha},Z)=g(\mathfrak{v}_{\beta},Z)=0$.
Further, for any $Y\in \mathfrak{v}_{\alpha}$ and $X\in \mathfrak{v}_{\beta}$ we have
$$
0=X\cdot g(Y,Z)=g([X,Y],Z)+g(Y,[X,Z])= g([X,Y],Z)-g(Y,L_Z(X)),
$$
Since $L_Z$ is bijective on $\mathfrak{v}_{\beta}$, this proves the first assertion of the lemma.
The second assertion follows from the first one and corollary \ref{orthog}.
\end{proof}

\begin{pred}\label{pr9}
For $\alpha, \beta \in I_1(Z)$, the equality $g({\alpha},\mathfrak{v}_\beta)=0$ holds under each of the following conditions:
$$
{\rm 1)}\, \alpha=\beta\,; \quad
{\rm 2)}\, \langle \alpha,\beta \rangle = 0\,;\quad
{\rm 3)}\, g(\mathfrak{v}_{\alpha},\mathfrak{v}_{\beta})=0\,.
$$
If $X\in \mathfrak{t}$ is such that $g(X,\alpha)=0$, then $g(X,\mathfrak{v}_{\alpha})=0$.
\end{pred}

\begin{proof}
If $\alpha=\beta$, then $g(\mathfrak{v}_{\alpha},\beta)=0$ by proposition \ref{pr3}.
Further, suppose that $\langle \alpha,\beta \rangle = 0$, which implies $[\alpha,\mathfrak{v}_{\beta}]=0$.
Since $g(\alpha,\alpha)=\langle \alpha, \alpha\rangle \, g(\alpha,Z)$ by proposition~\ref{pr3}, we get
$$
0=2g([X,\alpha],\alpha)=X\cdot g(\alpha,\alpha)=\langle \alpha, \alpha\rangle \,X\cdot g(\alpha,Z)=\langle \alpha, \alpha\rangle \,g(\alpha,[X,Z])
$$
for any $X\in \mathfrak{v}_{\beta}$, where $\langle \alpha,\beta \rangle = 0$
(note that $[\alpha,\mathfrak{v}_{\beta}]=0$ for $\langle \alpha,\beta \rangle = 0$).
Since $[Z,\mathfrak{v}_\beta]=\mathfrak{v}_\beta$, we get $g({\alpha},\mathfrak{v}_\beta)=0$.

Now, suppose that $g(\mathfrak{v}_{\alpha},\mathfrak{v}_{\beta})=0$. If $\langle \alpha,\beta \rangle = 0$ then we can apply 2).
Consider the case $\langle \alpha,\beta \rangle \neq 0$.

By proposition \ref{pr3} we have
$g(U_{\alpha},U_{\alpha})=g(V_{\alpha},V_{\alpha})=\frac{1}{\langle \alpha, \alpha\rangle} \,g(\alpha,\alpha)$.
From the equality $g(\mathfrak{v}_{\alpha},\mathfrak{v}_{\beta})=0$ we get
\begin{eqnarray*}
0&=&-2U_{\alpha}\cdot g(U_{\alpha},U_{\beta})=2g(U_{\alpha},[U_{\beta},U_{\alpha}])=U_{\beta}\cdot g(U_{\alpha},U_{\alpha}),\\
0&=&-2V_{\alpha}\cdot g(V_{\alpha},V_{\beta})=2g(V_{\alpha},[V_{\beta},V_{\alpha}])=V_{\beta}\cdot g(V_{\alpha},V_{\alpha}).
\end{eqnarray*}
Therefore,
{\small\begin{eqnarray*}
0&=&U_{\beta}\cdot g(U_{\alpha},U_{\alpha})=\frac{1}{\langle \alpha, \alpha \rangle} \Bigl(U_{\beta}\cdot g(\alpha,\alpha)\Bigr)=
\frac{2}{\langle \alpha, \alpha\rangle}\, g([U_{\beta},\alpha],\alpha)=
-\frac{2 \langle \alpha, \beta\rangle}{\langle \alpha, \alpha\rangle} g(V_{\beta},\alpha),\\
0&=&V_{\beta}\cdot g(V_{\alpha},V_{\alpha})=\frac{1}{\langle \alpha, \alpha\rangle} \Bigl(V_{\beta}\cdot g(\alpha,\alpha)\Bigr)=
\frac{2}{\langle \alpha, \alpha\rangle}\, g([V_{\beta},\alpha],\alpha)=
\frac{2 \langle \alpha, \beta\rangle}{\langle \alpha, \alpha \rangle} g(U_{\beta},\alpha).
\end{eqnarray*}}
Hence, $g({\alpha},\mathfrak{v}_\beta)=0$ in this case too.

Finally, if $g(X,\alpha)=0$ for $X\in \mathfrak{t}$, then for any $Y\in \mathfrak{v}_{\alpha}$ we get
$0=Y\cdot g(X,\alpha)=g([Y,X],\alpha)+g(X,[Y,\alpha])$. Since $[Y,X]\in \mathfrak{v}_{\alpha}$, we get
$g([Y,X],\alpha)\subset g(\mathfrak{v}_{\alpha},\alpha)=0$. Hence, $g(X,[Y,\alpha])=0$. Since $Y$ is an arbitrary in $\mathfrak{v}_{\alpha}$,
then $g(X,\mathfrak{v}_{\alpha})=0$.
\end{proof}

\begin{remark}
Note that $[\mathfrak{v}_{\alpha},\mathfrak{v}_{\beta}]=0$ implies $\langle {\alpha},{\beta} \rangle =0$.
Such roots ${\alpha}$ and ${\beta}$ are called {\it strongly orthogonal}.
\end{remark}

\begin{lemma}\label{le8}
If $g(\mathfrak{v}_{\alpha},\beta)=0$ for some $\alpha, \beta \in I_1(Z)$, then
$$
\langle \alpha,\beta \rangle \,g(U_{\alpha},U_{\alpha})=\langle \alpha,\beta \rangle \,g(V_{\alpha},V_{\alpha})=g(\alpha,\beta).
$$
In particular, if $\alpha,\beta \in I_1(Z)$ are such that $\langle \alpha,\beta \rangle=0$, then $g(\alpha,\beta)=0$.
\end{lemma}

\begin{proof} By direct calculations we get
$$
0=U_{\alpha}\cdot g(V_{\alpha},\beta)=g([U_{\alpha},V_{\alpha}],\beta)+g(V_{\alpha},[U_{\alpha},\beta])=
g(\alpha,\beta)-\langle \alpha,\beta \rangle g(V_{\alpha},V_{\alpha}).
$$
From this, proposition \ref{pr3} and proposition \ref{pr9} we get the lemma.
\end{proof}

\begin{pred}\label{newprort}
Put $\mathfrak{v}:=\bigoplus_{\alpha \in \Delta^+} \mathfrak{v}_{\alpha}$. Then the following assertions holds.

{\rm 1)} If $g(\mathfrak{t},\mathfrak{v})=0$ on
$(M,g)$, then every $X\in \mathfrak{g}_1$ has constant length on $(M,g)$;

{\rm 2)} If the roots of $\mathfrak{g}_1$ have one and the same length, and
$g(\mathfrak{t},\mathfrak{v}_{\alpha})=0$ for some $\alpha \in \Delta^+$, then  $g(\mathfrak{t},\mathfrak{v})=0$ on
$(M,g)$.
\end{pred}

\begin{proof} For every $X \in  \mathfrak{t}$ we have $g(X,\mathfrak{v}_{\alpha})=0$ for all $\alpha \in \Delta^+$.
This obviously implies that $g([X,[X,Y]],X)=0$ on $(M,g)$ for every $Y\in \mathfrak{g}$. Hence,
$X$ has constant length on $(M,g)$ by proposition \ref{pre1}, that proves 1).

Suppose that $g(\mathfrak{t},\mathfrak{v}_{\alpha})=0$ for some $\alpha \in \Delta^+$.
Let $\eta$ be an inner automorphism of $\mathfrak{g}_1$ (and of $\mathfrak{g}$ after the natural extension of the action),
such that $\eta(\mathfrak{t})=\mathfrak{t}$. Then
$g(\mathfrak{t},\eta(\mathfrak{v}_{\alpha}))=g(\eta(\mathfrak{t}),\eta(\mathfrak{v}_{\alpha}))=0$ by lemma \ref{s1}.
Restriction of $\eta$ to $\mathfrak{t}$ is an element $w$ of the Weil group $W$.
If $\beta=w(\alpha)$, this mean $g(\mathfrak{t},\mathfrak{v}_{\beta})=0$.
Since all $w\in W$ could be obtained by this manner (see section \ref{rsdsec}),
then $g(\mathfrak{t},\mathfrak{v}_{w(\alpha)})=0$ for all $w\in W$. Since the Weyl group acts transitively on the set of the roots
of a fixed length, then $g(\mathfrak{t},\mathfrak{v})=0$ for $\mathfrak{g}_1$ with the roots of one and the same length.
\end{proof}

\begin{remark}
The list of simple Lie algebras with the roots of one and the same length is the following:
$A_l=su(l+1),D_l=so(2l), e_6, e_7, e_8$.
\end{remark}

Now we are ready to finish the proof theorem \ref{theorem2}.

\begin{pred}\label{newpr6}
If the case {\rm 2)} of proposition \ref{newpr4} holds but {\rm 1)} does not hold, then
the pair $(\mathfrak{g}_1, \mathfrak{k})$ is one of the following irreducible Hermitian symmetric pair:

{\rm 1)} $(su(p+q), su(p)\oplus su(q)\oplus \mathbb{R})$, $p\geq q \geq 1$;

{\rm 2)} $(so(2n),su(n)\oplus \mathbb{R})$, $n\geq 5$;

{\rm 3)} $(so(p+2), so(p)\oplus \mathbb{R})$, $p \geq 5$;

{\rm 4)} $(sp(n),su(n)\oplus \mathbb{R})$, $n\geq 2$.
\end{pred}

\begin{proof} There are the following irreducible Hermitian symmetric pair (see \cite{Wolf64} or \cite{WolfBook}):
$(su(p+q), su(p)\oplus su(q)\oplus \mathbb{R})$, $p\geq q \geq 1$; $(so(2n),su(n)\oplus \mathbb{R})$, $n\geq 5$;
$(so(p+2), so(p)\oplus \mathbb{R})$, $p \geq 5$; $(sp(n),su(n)\oplus \mathbb{R})$, $n\geq 2$;
$(e_6, so(10)\oplus \mathbb{R})$; $(e_7, e_6 \oplus \mathbb{R})$.

We should exclude the last two pairs. Let us start with the pair $(e_6, so(10)\oplus \mathbb{R})$.
There are two non-compact roots for $e_6$: $\pi_1$ and $\pi_6$, but we may consider only first case, because the second is reduced to the first
by a suitable automorphism of the Lie algebra $e_6$. Hence $Z\in \mathfrak{t}$ is such that $\langle Z,\pi_i\rangle=0$ for $i=2,\dots,6$.
Up to multiple by a constant we get $Z=\frac{2}{3}(e_8-e_7-e_6)$, see section \ref{rsdsec}.

All root $\alpha=\frac{1}{2}\bigl(e_8-e_7-e_6+\sum_{i=1}^5 (-1)^{v_i} e_i\bigr)$ with even $\sum_{i=1}^5 v_i$ are in $I_1(Z)$.
Consider also $\beta=\frac{1}{2}\bigl(e_8-e_7-e_6+\sum_{i=1}^5 (-1)^{u_i} e_i\bigr)$ with even $\sum_{i=1}^5 u_i$, $\alpha \neq \beta$.
Clear that $\langle \alpha ,\beta \rangle=0$ if and only if $\sum_{i=1}^5 (-1)^{v_i+u_i}=-3$, that means that $(u_1,\dots,u_5)$ and $(v_1,\dots,v_5)$
are distinct exactly in four coordinates.
Put $(v_1,\dots,v_5)=(1,1,1,1,0)$ and let $(u_1,\dots,u_5)$ be such that $u_5=1$ and there are exactly one $0$ among $u_i$, $1\leq i \leq 4$.

If we consider also $\gamma:=\frac{1}{2}\bigl(e_8-e_7-e_6+\sum_{i=1}^5 e_i\bigr)$,
then $\langle \alpha ,\gamma \rangle=\langle \beta ,\gamma \rangle=0$ and
$g(\alpha,\mathfrak{v}_{\gamma})=g(\beta,\mathfrak{v}_{\gamma})=g(\alpha-\beta ,\mathfrak{v}_{\gamma})=0$ by proposition \ref{pr9}.
The vectors $\alpha-\beta$ are exactly all vectors of the form $e_5-e_i$,  $1\leq i \leq 4$.
Moreover, $g(\gamma,\mathfrak{v}_{\gamma})=0$ by proposition \ref{pr9}. Hence, $g(\alpha-\gamma,\mathfrak{v}_{\gamma})=0$.
Since $\gamma-\alpha=e_1+e_2+e_3+e_4$, then $g(e_i,\mathfrak{v}_{\gamma})=0$ for $1\leq i \leq 5$. Moreover, $g(e_8-e_7-e_6,\mathfrak{v}_{\gamma})=0$, because
$e_8-e_7-e_6=2\gamma-\sum_{i=1}^5 e_i$. Therefore, $g(\mathfrak{t},\mathfrak{v}_{\gamma})=0$ and
all $X\in \mathfrak{g}_1$ has constant length on $(M,g)$ by proposition~\ref{newprort}.
Hence, this case could be excluded.
\smallskip

Now, consider the pair $(e_7, e_6 \oplus \mathbb{R})$. There is only one non-compact root for $e_7$: $\pi_7$.
Hence $Z\in \mathfrak{t}$ is such that $\langle Z,\pi_i\rangle=0$ for $i=1,\dots,6$.
Up to multiple by a constant we get $Z=\frac{1}{2}(e_8-e_7+2e_6)$, see section \ref{rsdsec}.
It is clear that $e_8-e_7\in I_1(Z)$ and $e_6\pm e_i \in I_1(Z)$ for $1\leq i \leq 5$.
A root $\beta=\frac{1}{2}\bigl(e_8-e_7+\sum_{i=1}^6 (-1)^{v_i} e_i\bigr)$ with odd $\sum_{i=1}^6 v_i$ is in $I_1(Z)$ if and only if
$v_6=0$.

Note that $\langle e_8-e_7, e_6\pm e_i \rangle $=0 for $1\leq i \leq 5$. By proposition \ref{pr9},
we have $g(\alpha,\mathfrak{v}_{\alpha})=0$, where $\alpha=e_8-e_7$ and
$g(e_i,\mathfrak{v}_{\alpha})=0$ for all $1\leq i \leq 6$. Therefore, $g(\mathfrak{t},\mathfrak{v}_{\alpha})=0$ on $(M,g)$
and all $X\in \mathfrak{g}_1$ has constant length on $(M,g)$ by proposition~\ref{newprort}.
Consequently, this case also could be excluded.
\end{proof}
\medskip

\begin{proof}[Proof of theorem \ref{theorem2}] The proof follows immediately from propositions \ref{newpr4}, \ref{newpr5}, and~\ref{newpr6}.
\end{proof}

\section{Further study of the Hermitian case}\label{fshcsec}

In this section we describe some  important properties of Killing vector fields of constant length $Z=Z_0+Z_1$
for the following Hermitian symmetric pairs:
\begin{eqnarray*}
(su(p+q), su(p)\oplus su(q)\oplus \mathbb{R}),\, p\geq q \geq 1;&& (so(2n),su(n)\oplus \mathbb{R}),\, n\geq 5;\\
(so(p+2), so(p)\oplus \mathbb{R}),\,\,\, p \geq 5; && (sp(n),su(n)\oplus \mathbb{R}),\,\,\,\, n\geq 2.
\end{eqnarray*}
This information could be useful in the forthcoming study of homogeneous Riemannian manifolds with nontrivial Killing vector fields of constant length.
\smallskip

We start with the following question.

\begin{vopros}\label{vop1}
Let $Z=Z_0+Z_1$, be a Killing field of constant length on $(M,g)$, where
$Z_0\in \mathfrak{c}$, $Z_1 \in \mathfrak{g}_1$, $Z_1 \neq 0$ as in theorem \ref{theorem2}.
Is it true that $Z_1$ is also a Killing field of constant length on $(M,g)$?
\end{vopros}

We know that the answer to this question is negative in general (see example \ref{exam5} in section \ref{examsec} and proposition \ref{newpr7} below).
But the answer is positive for some types of simple Lie algebras~$\mathfrak{g}_1$.
Recall also that the Killing field $Z_0\in \mathfrak{c}$ always has constant length on $(M,g)$.

\begin{pred}\label{newpr6.5}
If, in the above notation, there is an inner automorphism of $\mathfrak{g}_1$,
such that $\psi(Z_1)=-Z_1$, then both Killing fields $Z_0-Z_1$ and $Z_1$ have constant length on $(M,g)$.
This is the case when the Weyl group $W$ of $\mathfrak{g}_1$ contains the map $-\Id$.
\end{pred}

\begin{proof}
Indeed, $Z_0-Z_1=Z_0+\psi(Z_1)$ has constant length (by lemma \ref{le1}) with the same $\mathfrak{k}$ and $\mathfrak{m}$ as $Z=Z_0+Z_1$ has.
Since $g(Z_0\pm Z_1, \mathfrak{m})=0$ by corollary \ref{orthog}, we
get $g(Z_0, \mathfrak{m})=0$ and $g(Z_1, \mathfrak{m})=0$. Hence, $Z_1$ has constant length by
proposition \ref{pre1}, because $\im(L_{Z})=\im(L_{Z_1})=\mathfrak{m}$.

Finally, if $W$ contains the map $-\Id$, then there is an inner automorphism $\eta$ of $\mathfrak{g}$ such that $\eta(Z_1)=-Z_1$
by properties of the Weyl group $W$ in section \ref{rsdsec}.
\end{proof}

\begin{remark}
Note that the Weyl group $W$  of a simple Lie algebra $\mathfrak{g}$ does not contain the map $-\Id$,
if and only if  $\mathfrak{g}$ is one of the following Lie algebras
{\rm(}see e.g. corollary 5.6.3 in \cite{Bourbaki}{\rm):}
$su(l+1)=A_l,\, l\geq 2, \,\, so(4k+2)=D_{2k+1}, \, k\geq 1, \,\, e_6$.
\end{remark}

\begin{pred}\label{newpr7}
If, in the above notation, $Z_1$ has constant length on $(M,g)$, then $g(Z_0,\mathfrak{g}_1)=0$ on $(M,g)$.
In particular, if  $\mathfrak{g}=\mathfrak{c}\oplus \mathfrak{g}_1$ and $M$ is simply connected in addition, then
$Z_0=0$.
\end{pred}

\begin{proof}
Since both $Z=Z_0+Z_1$ and $Z_1$ has constant length on $(M,g)$, then $g(Z,\mathfrak{m})=g(Z_1,\mathfrak{m})=0$,
where $\mathfrak{m}$ is the $\langle \cdot,\cdot \rangle$-orthogonal complement to $\mathfrak{k}$
(the centralizer of $Z_1$ and $Z$ in $\mathfrak{g}_1)$, see corollary \ref{orthog}.
Hence, $g(Z_0,\mathfrak{m})=0$ and $0=\mathfrak{m}\cdot g(Z_0,\mathfrak{m})=g(Z_0,[\mathfrak{m},\mathfrak{m}])$.
Since $\mathfrak{m}+[\mathfrak{m},\mathfrak{m}]$ is an ideal in $\mathfrak{g}_1$ by lemma \ref{vslem1}, then
$\mathfrak{m}+[\mathfrak{m},\mathfrak{m}]=\mathfrak{g}_1$ and consequently, $g(Z_0,\mathfrak{g}_1)=0$.
The second assertion follows from the fact that $\mathfrak{g}_1$ acts transitively on simply connected $M$
(see e.g. proposition 9 on p.~94 in \cite{On1994}).
\end{proof}
\medskip

Now, we are going to discuss four Hermitian pairs in theorem \ref{theorem2} in more details.
\bigskip

{\bf  1. The pair } $(su(p+q), su(p)\oplus su(q)\oplus \mathbb{R})$, $p\geq q \geq 1$. For the Lie algebra $su(p+q)$,
all simple roots $\pi_i$, $1\leq i \leq p+q-1$, are non-compact, see section \ref{rsdsec}. Let us fix $i=p$.
Then $Z_1$ is $\langle \cdot,\cdot\rangle$-orthogonal to all $\pi_j=e_j-e_{j+1}$, $j\neq p$. Up to multiplication by a constant,
we have
\begin{equation}\label{eqsu1}
Z_1=\frac{1}{p+q}\Bigl( q\,(e_1+e_2+\cdots+e_p)-p\,(e_{p+1}+\cdots+e_{p+q}) \Bigr).
\end{equation}
A root $\alpha=e_i-e_j$ is in $I_1(Z)\subset \Delta^+$ if and only if $i\leq p$ and $j\geq p+1$ (in such a case $\langle Z, \alpha\rangle=1$).

Now, suppose that $q\geq 2$ and consider pairwise distinct indices $i,j,k,l$. Using properties of the Weyl group $W$
of $su(p+q)$ (which is the permutation group of the vectors $e_i$, $1\leq i\leq p+q$), we can choose a Killing fields of constant length
of the type $Z_0+w(Z_1)$, $w\in W$, (see lemma \ref{pr1.1})
in such a way, that $\alpha,\beta \in I_1(Z_0+w(Z_1))$, where $\alpha=|e_i-e_j|$ and $\beta=|e_k-e_l|$.
Since $[\mathfrak{v}_\alpha,\mathfrak{v}_\beta]=0$ and $\langle \alpha, \beta \rangle=0$, then we get
\begin{equation}\label{eqsu2}
g(\mathfrak{v}_{\alpha},\mathfrak{v}_\beta)=g(\mathfrak{v}_{\alpha},\beta)=g(\alpha,\mathfrak{v}_\beta)=g(\alpha,\beta)=0
\end{equation}
on $(M,g)$ by lemmas  \ref{le2}, \ref{le8} and proposition \ref{pr9}. In particular, there are $q$ roots $\alpha_i$, $1\leq i \leq q$
such that $g(\alpha_i,\alpha_j)=0$ on $(M,g)$ for $i\neq j$.

Note also that $g(\alpha,\alpha)=\langle \alpha, \alpha\rangle \, g(\alpha,Z)=2g(\alpha,Z)$ by proposition \ref{pr3}
for any $\alpha =e_i-e_j\in I_1(Z)$. Now, we suppose that $Z_0=0$. Then, using the explicit form of $Z=Z_1$ (\ref{eqsu1}) and (\ref{eqsu2}) we get
{\small\begin{eqnarray*}
\frac{p+q}{2}\, g(e_i-e_j,e_i-e_j)=(p+q)g(Z,e_i-e_j)\\
=g\left((e_i-e_j)+\sum_{k\leq p,k\neq i} (e_k-e_j)+  \sum_{l> p,l\neq j} (e_i-e_l),e_i-e_j \right)
\end{eqnarray*}}
Let us fix $r\leq p$ and $s>p$ such that $r\neq i$ and $s\neq j$, then using (\ref{eqsu2}) again,  we get
{\small\begin{eqnarray*}
\frac{p+q}{2}\, g(e_i-e_j,e_i-e_j)=(p+q)g(Z,e_i-e_j)=g(e_i-e_j,e_i-e_j)\\
+ g\left(\sum_{k\leq p,k\neq i} \bigl((e_k-e_r)+(e_r-e_j)\bigr)+  \sum_{l> p,l\neq j} \bigl((e_i-e_s)+(e_s-e_l)\bigr),e_i-e_j \right)\\
=g(e_i-e_j,e_i-e_j)+(p-1)g(e_r-e_j,e_i-e_j)+(q-1)g(e_i-e_s,e_i-e_j),
\end{eqnarray*}}
\!\!that is equivalent to
{\small\begin{eqnarray}\label{eqsu3}
(p+q-2) g(e_i-e_j,e_i-e_j)=2(p-1)g(e_r-e_j,e_i-e_j)+2(q-1)g(e_i-e_s,e_i-e_j).
\end{eqnarray}}
\!\!Using the Weyl group as above, we see that (\ref{eqsu3}) is fulfilled for every distinct $i,j,r,s$.
Interchanging $i$ and $j$, as well as $r$ and $s$, in (\ref{eqsu3}), we get
{\small\begin{eqnarray*}
(p+q-2) g(e_i-e_j,e_i-e_j)=2(q-1)g(e_r-e_j,e_i-e_j)+2(p-1)g(e_i-e_s,e_i-e_j).
\end{eqnarray*}}
\!\!Subtracting (\ref{eqsu3}) from the this equality, we get $(p-q)g(e_r+e_s-e_i-e_j,e_i-e_j)=0$, that
implies $g(e_r-e_j,e_i-e_j)=g(e_i-e_s,e_i-e_j)$ for $p\neq q$. From this and (\ref{eqsu3}) we get
$$
(p+q-2) g(e_i-e_j,e_i-e_j)=2(p+q-2)g(e_i-e_s,e_i-e_j),
$$
that implies $g(e_i-e_j,e_i-e_j-2(e_i-e_s))=0$. Put $\alpha=e_i-e_j$, then we get
$g(\mathfrak{v}_{\alpha},\alpha)=0$ and $g(\mathfrak{v}_{\alpha},e_i-e_s)=0$ for any $s\neq i$ by proposition \ref{pr9}. Therefore,
$g(\mathfrak{v}_{\alpha},\mathfrak{t})=0$ and every $X\in \mathfrak{g}_1=su(p+q)$ has constant length on $(M,g)$ by
proposition~\ref{newprort}.
Then we get

\begin{pred}\label{supred} If $Z=Z_1$ {\rm(}i.e. $Z_0=0${\rm)} for the pair  $(su(p+q), su(p)\oplus su(q)\oplus \mathbb{R})$ and $p\neq q\geq 2$, then
every $X\in \mathfrak{g}_1=su(p+q)$ has constant length on $(M,g)$.
\end{pred}

Note also that for $p=q$ there is $w\in W$ such that $w(Z_1)=-Z_1$, hence, $Z_1$ is a Killing field of constant length on $(M,g)$
by proposition \ref{newpr6.5}.
\bigskip

{\bf  2. The pair $(so(2n),su(n)\oplus \mathbb{R})$, $n\geq 5$.} For the Lie algebra $so(2n)$,
the are three non-compact roots, $\pi_1$ and $\pi_{n-1}$, and $\pi_n$, see section \ref{rsdsec}.
For this pair we should choose $\pi_n$
($\pi_{n-1}$ lead to the same result due to a suitable automorphism of $so(2n)$, and $\pi_1$ corresponds to another pair).
Then $Z_1$ is $\langle \cdot,\cdot\rangle$-orthogonal to all $\pi_i=e_i-e_{i+1}$, $i<n$. Up to multiplication by a constant,
we have
\begin{equation}\label{eqsonn1}
Z_1=\frac{1}{2}\Bigl(e_1+e_2+\cdots+e_{n-1} +e_n\Bigr).
\end{equation}
Clear that $I_1(Z)\subset \Delta^+$ consists of the root $\alpha=e_i+e_j$, $1\leq i<j \leq n$, and $\langle Z, \alpha\rangle=1$ for such $\alpha$.

Let us choose distinct indices $i,j,k,l$ and put $\alpha=e_i+e_j$, $\beta=e_k+e_l$.
Since $[\mathfrak{v}_\alpha,\mathfrak{v}_\beta]=0$ and $\langle \alpha, \beta \rangle=0$, then we get
\begin{equation}\label{eqsonn2}
g(\mathfrak{v}_{\alpha},\mathfrak{v}_\beta)=g(\mathfrak{v}_{\alpha},\beta)=g(\alpha,\mathfrak{v}_\beta)=g(\alpha,\beta)=g(e_i+e_j,e_k+e_l)=0
\end{equation}
on $(M,g)$ by lemmas  \ref{le2}, \ref{le8} and proposition \ref{pr9}.
The Weyl group $W$ permutes $e_i$, $1\leq i \leq n$, and changes the signs of even numbers of the basic vectors.
Using $Z_0+w(W)$ for a suitable $w\in W$ instead of $Z=Z_0+Z_1$, we easily get that $g(e_i+e_j,e_k-e_l)=0$ too (here we have used that $n\geq 5$).
Therefore, $g(e_i+e_j,e_k)=0$. Using the Weyl group one more time, we get $g(e_i,e_k)=0$ on $(M,g)$ for any $i\neq k$.

\begin{pred}\label{sonnpred} If a Killing field $Z=Z_0+Z_1$ has constant length on $(M,g)$ for the pair $(so(2n),su(n)\oplus \mathbb{R})$, then
there exist $n$ pairwise commuting $U_i\in so(2n)$, $1\leq i \leq n$, such that $\langle U_i, U_j \rangle =\delta_{ij}$,
$Z_1=c\cdot \sum_{i=1}^n{U_i}$ for some positive $c\in \mathbb{R}$
and $g(U_i,U_j)=0$ for $i\neq j$ on $(M,g)$.
\end{pred}

Note that $\delta_{ij}$ above means the Kronecker delta, i.e. $\delta_{ij}=1$ for $i=j$ and $\delta_{ij}=0$ for $i\neq j$.
It should be noted also that for even $n$, the Weyl group contained $-\Id$, Therefore, $Z_1$ is also Killing field of constant length on $(M,g)$ in this case
by proposition \ref{newpr6.5}.
\bigskip

{\bf  3.  The pair $(so(p+2), so(p)\oplus \mathbb{R})$, $p \geq 5$.} In this case we should distinguish the subcase $p=2l-2$ for some integer $l\geq 4$
and the subcase $p=2l-1$ for some integer $l\geq 3$.

If $so(p+2)=so(2l)$, we should consider non-compact root $\pi_1=e_1-e_2$,
see section~\ref{rsdsec}.
Then $Z_1$ is $\langle \cdot,\cdot\rangle$-orthogonal to all $\pi_i$, $2\leq i\leq n$. Up to multiplication by a constant,
we have
\begin{equation}\label{eqson1}
Z_1=e_1.
\end{equation}
Clear that $I_1(Z)\subset \Delta^+$ consists of the root $\alpha=e_1\pm e_i$, $2\leq i \leq n$, and $\langle Z, \alpha\rangle=1$ for such $\alpha$.

If $so(p+2)=so(2l+1)$, a unique non-compact root in $so(2l+1)$ is also $\pi_1=e_1-e_2$ and we have $Z$ as in (\ref{eqson1}).
Note, however, that in this subcase $Z_1$ is also the root.
$I_1(Z)\subset \Delta^+$ consists of the roots $\alpha=e_1\pm e_i$, $2\leq i \leq n$, and $e_1$.

The next arguments work both $so(p+2)=so(2l)$ and $so(p+2)=so(2l+1)$.
Let $F_{i,j}$ be a $(p+2)\times(p+2)$-matrix $(a_{\,t,s})$ with all zero entries excepting $a_{i,j}=-1$ and $a_{j,i}=1$.
We will identify $F_{2i-1,2i}$ with $e_i$ in $so(p+2)$, see e.g. chapter 8 in~\cite{Bes}. Hence $F_{1,2}=e_1=Z_1$.
If we take $a=\diag\left(\!\!\left(%
\begin{array}{cc}
  0 & 1 \\
  1 & 0 \\
\end{array}%
\right)\!\!,1,\dots,\!1\!\!\right)\in SO(p+2)$, then $\Ad(a)(F_{1,2})=-F_{1,2}$. Therefore, $Z_1=F_{1,2}$ is also constant length on $(M,g)$
by proposition~\ref{newpr6.5}.
In what follows, we assume that $Z=Z_1=F_{1,2}$.

Now, for all $a\in SO(p+2)$ and all $c\in \mathbb{R}$, the vector $\Ad(a)(cF_{1,2})$ has constant length on $(M,g)$ by lemma \ref{s1}.
From this observation we get that all $X\in \mathfrak{f}$ have constant length on $(M,g)$,
where $\mathfrak{f}$ is the subalgebra of $so(p+2)$, spanned by the Killing vector fields $F_{1,2}$, $F_{2,3}$, and $F_{1,3}$.

For every $X,Y \in \mathfrak{f}$ we have $2g(X,Y)=g(X+Y,X+Y)-g(X,X)-g(Y,Y)=\const$, because all $X\in \mathfrak{f}$
have constant length. Hence,  $2\nabla_XY=[X,Y]\in \mathfrak{f}$ by lemma \ref{lemma3}.
This means that the distribution $\mathfrak{f}$ is autoparallel on $(M,g)$.

Moreover, $R(X,Y)Y=-\nabla_Y\nabla_Y X=-\frac{1}{4}[Y,[Y,X]]$ by lemma \ref{lemma5}.
But it is easy to see that $[Y,[Y,X]]=-X$ for any $X,Y\in \mathfrak{f}$ such that $\langle X,X\rangle =\langle Y,Y\rangle =\langle F_{1,2},F_{1,2}\rangle$ and
$\langle X,Y\rangle=0$ ($\mathfrak{f}$ is isomorphic to $so(3)$). Therefore, $R(X,Y)Y=\frac{1}{4}X$.
Moreover, $g(X,Y)=0$ by proposition \ref{pre1} and $g(Y,Y)=g(F_{1,2},F_{1,2})$ by lemma \ref{s1}.
Hence, the sectional curvature of the plane spanned by $X$ an $Y$ (at every point $x\in M$) is equal
$$
\frac{g(R(X,Y)Y,X)}{g(X,X)g(Y,Y)-g(X,Y)^2}=\frac{1}{4g(Y,Y)}=\frac{1}{4g(F_{1,2},F_{1,2})}=\const.
$$

It is clear, that the subalgebra $\psi(\mathfrak{f})$ in $so(p+2)$, where $\psi$ is any inner automorphism of $so(p+2)$,
also consist of Killing vector fields of constant length. In particular, any $F_{i,j}\in so(p+2)$, $i\neq j$, has constant length on $(M,g)$.
From this we get
$$
0=F_{k,j} \cdot g(F_{i,k},F_{i,k})=2 g ([F_{k,j},F_{i,k}],F_{i,k})=2g(F_{i,j},F_{i,k}),\quad i\neq k.
$$
therefore, the linear subspace $P_i:=\Lin(F_{i,j}\,|\, j\neq i)$ in $so(p+2)$ consists of Killing vector fields of constant length.
Therefore, we get

\begin{pred}\label{sonpred}
If a Killing vector field $Z=Z_0+Z_1$ has constant length on $(M,g)$ for the pair $(so(p+2),so(p)\oplus \mathbb{R})$,
then the following assertions hold:

{\rm 1)} The Killing field $Z_1$ has constant length too;

{\rm 2)} There is a $3$-dimensional Lie subalgebra $\mathfrak{f}$ in $so(p+2)$, that consists of Killing fields of constant length and
such that all orbits of the subgroup $\exp(\mathfrak{f})\subset \exp(\mathfrak{g})=G$ in $(M,g)$ are totally geodesic and have
{\rm(}one and the same{\rm)} constant positive curvature;

{\rm 3)} There are $(p+1)$-dimensional linear subspaces in $so(p+2)$, that consist of Killing vector fields of constant length.
\end{pred}

Note that linear subspaces, that consist of Killing vector fields of constant length are called {\it Clifford-Killing spaces} in \cite{BerNik2009},
where Clifford-Killing spaces for Euclidean spheres were studied in particular. The authors of \cite{XuWolf2014} use the term
{\it Clifford-Killing vector fields} for Killing vector fields of constant length.
\bigskip

{\bf  4.  The pair $(sp(n),su(n)\oplus \mathbb{R})$, $n\geq 2$.} A unique non-compact root in
$sp(n)$ is $\pi_n=2e_n$. Therefore, $Z_1$ is $\langle \cdot,\cdot\rangle$-orthogonal to all $\pi_i=e_i-e_{i+1}$, $i<n$. Up to multiplication by a constant,
we have
\begin{equation}\label{eqspnn1}
Z_1=\frac{1}{2}\Bigl(e_1+e_2+\cdots+e_{n-1} +e_n\Bigr).
\end{equation}
Clear that $I_1(Z)\subset \Delta^+$ consists of the roots $2e_i$, $1\leq i \leq n$, and
$e_i+e_j$, $1\leq i<j \leq n$.
Let us choose distinct indices $i,j$ and put $\alpha=2e_i$, $\beta=2e_j$.
Since $[\mathfrak{v}_\alpha,\mathfrak{v}_\beta]=0$ and $\langle \alpha, \beta \rangle=0$, then we get
\begin{equation}\label{eqspnn2}
g(\mathfrak{v}_{\alpha},\mathfrak{v}_\beta)=g(\mathfrak{v}_{\alpha},\beta)=g(\alpha,\mathfrak{v}_\beta)=g(\alpha,\beta)=4g(e_i,e_j)=0
\end{equation}
on $(M,g)$ by lemmas  \ref{le2}, \ref{le8} and proposition \ref{pr9}.

Since the Weyl group $W$ for $sp(n)$ contains the map $-\Id$,
then the Killing fields $Z_1$ also has constant length on $(M,g)$
by proposition \ref{newpr6.5}.

\begin{pred}\label{sppred} If a Killing field $Z=Z_0+Z_1$ has constant length on $(M,g)$ for the pair $(sp(n),su(n)\oplus \mathbb{R})$, then
the Killing field $Z_1$ also has constant length on $(M,g)$ and
there exist $n$ pairwise commuting $U_i\in sp(n)$, $1\leq i \leq n$, such that $\langle U_i, U_j \rangle =\delta_{ij}$,
$Z_1=c\cdot \sum_{i=1}^n{U_i}$ for some positive $c\in \mathbb{R}$
and $g(U_i,U_j)=0$ for $i\neq j$ on $(M,g)$.
\end{pred}

\section{Unsolved questions}

Results of this paper can be used for the classification of homogeneous Riemannian manifold with Killing vector fields of some special kinds.
We state some unsolved questions and problems in this direction.

\begin{problem}\label{pr1}
Classify  homogeneous Riemanninan spaces $(G/H,g)$ with nontrivial Killing vector fields of constant length, where $G$ is simple.
\end{problem}
Recall, that normal homogeneous Riemanninan spaces $(G/H,g)$ with Killing fields of constant length, where $G$ is simple, are classified
in \cite{XuWolf2014}. On the other hand, the set of $G$-invariant metrics on a space $G/H$ could have any dimension even for
simple~$G$.

\begin{vopros}\label{vop2}
Let $(M, g)$ be a Riemannian homogeneous manifold, $G$ be its full connected isometry group, and
$Z \in \mathfrak{g}$ be a Killing vector field of constant length on $(M, g)$.
Does~$\mathfrak{k}$, the centralizer of $Z$ in $\mathfrak{g}$, acts transitively on $M$?
\end{vopros}
Note, that for symmetric spaces  $(M=G/H, g)$ we have an affirmative answer to this question (see lemma 3 in \cite{BerNik2008n}).
Of course, question \ref{vop2} is interesting even under an additional assumption that the group $G$ is simple.
\medskip

There is a natural generalization of normal homogeneous spaces.
Recall, that a Riemannian manifold $(M=G/H,g)$, where $H$ is a compact subgroup
of a Lie group $G$ and $g$ is a $G$-invariant Riemannian metric,
is called a geodesic orbit space
if any geodesic $\gamma$ of $M$ is an orbit of
1-parameter subgroup of the group $G$, detailed information on this class of homogeneous Riemannian manifolds one can find e.g. in
\cite{AA, AlNik, BerNikBook, KV, Ta}.
It is known that all normal homogeneous space are geodesic orbit. The following problem is natural.

\begin{problem}\label{pr2}
Classify  geodesic orbit Riemannian space with nontrivial Killing vector fields of constant length.
\end{problem}

Recall the following result of \cite{Nik2013}:
If $(M=G/H,g)$ is a geodesic orbit Riemannian space, $\mathfrak{g}$ is its Lie algebra of Killing fields,
and $\mathfrak{a}$ is an abelian ideal in $\mathfrak{g}$, then every $X\in \mathfrak{a}$
has constant length on $(M,g)$.
Hence it suffices to study only homogeneous spaces $G/H$ with semisimple $G$ in problem \ref{pr2}.
\medskip

Note that the examples from section \ref{examsec} cover all possibilities in theorem \ref{theorem2},
excepting the case {\rm 3)}, where we have examples only for specific values of $p$.

\begin{vopros}\label{vop3}
Is there a Riemannian homogeneous space $(M=G/H, g)$  with a Killing vector field of constant length
$Z \in \mathfrak{g}$, corresponded to the case {\rm 3)} in theorem~\ref{theorem2}, i.e.
$(\mathfrak{g}_1,\mathfrak{k})=(so(p+2), so(p)\oplus \mathbb{R})$, for $p=6$ and $p\geq 8$?
\end{vopros}

For $p=5$ and $p=7$ see examples \ref{exam6} and \ref{exam7}.
Note that proposition \ref{sonpred} gives serious restrictions on a Riemannian manifold  $(M=G/H, g)$ as in question \ref{vop3}.

\bigskip
\bigskip

\end{document}